\documentclass{mc}
\renewcommand\datereceived{}
\renewcommand\dateaccepted{}

\usepackage{amsmath, amsthm, amssymb, verbatim, bbm, color, graphics,  geometry}

\usepackage{float}
\usepackage{tabularx}
\usepackage{multirow} 
\usepackage{booktabs}
\usepackage{makecell}
\usepackage{mathrsfs}
\usepackage{caption}
\usepackage{subcaption}
\usepackage{rotating}
\usepackage{braket}
\usepackage{listings}
\usepackage{enumitem}
\usepackage{mathrsfs} 

\usepackage{booktabs} 
\usepackage{placeins}
\usepackage[normalem]{ulem}

\usepackage{xcolor,sectsty}
\definecolor{astral}{RGB}{46,116,181}
\definecolor{darkslategray}{rgb}{0.18, 0.31, 0.31}
\definecolor{warmblack}{rgb}{0.0, 0.46, 0.36}
\sectionfont{\color{astral}\fontsize{12}{12}\selectfont}
\subsectionfont{\color{astral}\fontsize{11}{12}\selectfont}
\subsubsectionfont{\color{astral}\fontsize{11}{11}\selectfont}

\usepackage{hyperref}
\usepackage{bm}

\begin{document}
	
    	\markboth{D. Gupta {\em et al.}}{QPINNs for Solving Integro and Fractional PDEs}
	\title{{\color{warmblack} Quantum Physics-Informed Neural Networks for Solving Integro and Fractional PDEs}}

	\author[D. Gupta {\em et al.}]{{\bf Deepak Gupta}\affil{1} and {\bf Ratikanta Behera}\affil{1}\corrauth}
 
	\address{\affilnum{1}\ {Department of Computational and Data Sciences, Indian Institute of Science, Bangalore, 560012, India}}
	
	\emails{{\tt deepakgupta1@iisc.ac.in} (D. Gupta), {\tt ratikanta@iisc.ac.in} (R. Behera)}
	%

\begin{abstract}
Quantum neural networks have emerged as powerful models for approximating nonlinear functions. Yet their use in solving integro-differential equations (IDEs) and fractional integro-partial differential equations (FIPDEs), which involve inherently nonlocal operators, remains unexplored. This work introduces a quantum physics-informed neural network (QPINN) framework that combines a quantum neural network with the governing equations of general nonlinear IDEs and FIPDEs. The proposed quantum network uses an affine feature map and variational quantum circuits to produce trial solutions with explicit trigonometric structure. We prove a quantitative $L^{2}(\mu)$ universal approximation theorem for this architecture, achieving a convergence rate of $\mathcal{O}(n^{-1/2})$. This extends classical Fourier approximation theory to quantum circuits for physics-informed learning. We propose two QPINN variants: the numerical-quadrature QPINN (N-QPINN), which handles nonlocal integrals and fractional operators via high-order numerical quadrature while computing local derivatives through automatic differentiation of quantum trial solutions; and the auxiliary-function QPINN (A-QPINN), which eliminates numerical quadrature by introducing auxiliary variables that reformulate each integro-differential equation as an equivalent coupled system of partial differential equations, enabling a multi-output quantum neural network to simultaneously represent the solution and its associated variables. A series of numerical experiments demonstrates that the proposed QPINN framework accurately captures the behavior of nonlinear IDEs and FIPDEs and outperforms classical physics-informed neural networks.
\end{abstract}


\maketitle

\section{Introduction}\label{sec_intro}
The rapid evolution of quantum computing has ushered in a new paradigm for tackling high-dimensional computational challenges in science and machine learning. By leveraging quantum mechanical principles such as superposition, entanglement, and unitary evolution, quantum processors offer computational behaviors that are fundamentally distinct from those of classical systems. Early algorithmic breakthroughs, including Shor’s factoring algorithm and Grover’s search, established the potential for exponential and quadratic quantum speedups.  More recently, advances such as the Harrow, Hassidim, and Lloyd (HHL) algorithm for solving linear systems~\cite{harrow2009quantum} have further revealed the potential of quantum devices for numerical analysis and scientific computing. Complementing these theoretical advances, the emergence of variational quantum algorithms (VQAs)~\cite{bravo2023variational} and quantum machine learning (QML) models~\cite{biamonte2017quantum} has created a practical pathway for hybrid quantum-classical computation on noisy intermediate-scale quantum (NISQ) hardware. Notably, quantum neural networks (QNNs) exhibit expressive properties fundamentally distinct from those of classical neural networks, including efficient representation of highly oscillatory, multiscale, and high-dimensional structures~\cite{abbas2021power, schuld2021effect}. These capabilities make QNNs promising candidates for enhancing scientific machine learning models.

On the classical computing side, there has been a renaissance in machine learning driven by deep neural networks. This resurgence has been enabled by advances in optimization techniques, software ecosystems, and specialized hardware accelerators, allowing neural networks to efficiently process continuous vectors and tensors in high-dimensional spaces. In particular, the development of physics-informed neural networks (PINNs)~\cite{raissi2019physics} has reshaped computational approaches for solving partial differential equations (PDEs). By embedding governing physical laws directly into the training objective, PINNs eliminate the need for labeled solution data and provide a mesh-free approximation framework through automatic differentiation. PINNs have achieved notable success in fields ranging from fluid dynamics~\cite{Cai2022Review, He2020AWRes, Mao2020HS, Zhu2021CM} and heat transfer~\cite{Cai2021Heat, Laubscher2021} to biomedical engineering~\cite{Arzani2021, SahliCostabal2020} and materials science~\cite{Chen2020OE, Goswami2020, Yin2021Thrombus}. Despite these advantages, standard PINNs often struggle when applied to stiff dynamics, sharp gradients, multiscale phenomena, and high-dimensional parameter spaces~\cite{Ying2019,karniadakis2021physics}. Increasing network expressivity can partially address these issues, but typically at the cost of higher computational expense, increased risk of overfitting, and unstable optimization~\cite{Rahaman2019}. These difficulties are more severe for integro-differential equations (IDEs) and fractional integro-partial differential equations (FIPDEs), which involve nonlocal operators that encode memory effects, anomalous diffusion, and long-range interactions. Several PINN extensions have attempted to mitigate these challenges. Frameworks such as DeepXDE~\cite{Lulu-DeepXDE2021}, fractional PINNs (fPINNs)~\cite{Pang2019fPINNs}, and auxiliary PINNs (A-PINNs)~\cite{Yuan2022APINN} introduce discretization techniques or auxiliary variables to better accommodate nonlocal and fractional operators. While these approaches alleviate some computational difficulties, they continue to suffer from the inherent expressivity of classical neural architectures and their difficulty in representing multiscale nonlocal behavior.These persistent limitations motivate the exploration of quantum machine learning models~\cite{Benedetti2019, PhysRevASchuld2020}, whose enhanced expressive capacity and ability to generate rich families of oscillatory features offer a compelling alternative for learning complex solutions~\cite{abbas2021power}.

Recent theoretical advances, including universal approximation theorems(UAT) for QNNs~\cite{goto2022universal, GononUAT2025} and analyses of quantum feature embeddings~\cite{lloyd2020quantum}, provide rigorous justification for applying QNNs to scientific computation. Leveraging entanglement, quantum parallelism, and high-frequency feature generation, QNNs offer expressive advantages over classical networks in regimes where PINNs often fail~\cite{du2021learnability, abbas2021power, schuld2021effect}. This has led to the emergence of physics-informed quantum neural networks (PIQNNs), with early demonstrations by Markidis et al.~\cite{markidis2022physics} and subsequent hybrid or fully quantum variants addressing nonlinear PDEs~\cite{trahan2024quantum, sedykh2024hybrid, panichi2025quantum, farea2025qcpinn, leong2025hybrid}. However, existing PIQNN approaches focus exclusively on PDEs governed by local differential operators. However, the application of quantum-enhanced learning to equations governed by nonlocal operators, specifically nonlinear integro-differential equations (IDEs) and fractional integro-partial differential equations (FIPDEs), remains unexplored.

In this work, we establish the universal approximation capabilities of a specific class of QNNs. We explicitly construct a parameterized quantum circuit and prove that an appropriate choice of circuit parameters enables accurate function approximation. More precisely, we show that a QNN with $\mathcal{O}(\varepsilon^{-2})$ trainable parameters and $\mathcal{O}(\lceil \log_2(\varepsilon^{-1}) \rceil)$ qubits suffices to approximate functions with integrable Fourier transforms to within a prescribed error tolerance $\varepsilon>0$. Notably, improving the approximation accuracy requires only a logarithmic growth in the number of qubits, demonstrating the favorable scalability of the proposed quantum architecture. This approximation result reveals that QNNs naturally generate finite trigonometric feature expansions with explicit $\mathcal{L}^2$-error guarantees. These expansions are closely analogous to classical Fourier feature representations, but are governed by learnable quantum parameters arising from the underlying quantum circuit architecture. This property provides a rigorous theoretical foundation for employing QNNs as expressive function approximators in scientific machine learning. Building on this approximation theory, we develop the first quantum physics-informed neural network (QPINN) framework designed explicitly for nonlinear integro-differential equations (IDEs) and fractional integro-partial differential equations (FIPDEs). Leveraging the trigonometric feature structure induced by QNNs, we introduce two complementary formulations:
\begin{enumerate}
    \item Numerical-Quadrature QPINN (N-QPINN): A hybrid strategy that evaluates integral operators using high-order numerical quadrature while preserving automatic differentiation for local differential operators.
    \item Auxiliary-Function QPINN (A-QPINN): A quadrature-free formulation that replaces nonlocal integral operators with auxiliary functions satisfying differentiable consistency relations, enabling mesh-free and fully differentiable learning of nonlocal behavior.
\end{enumerate}
Together, these formulations provide a unified quantum framework for addressing a broad class of nonlinear integro-differential and fractional integro-partial differential equations. To the best of our knowledge, this work represents the first systematic development of QPINN methodologies for such nonlocal models. The main contributions of this work can be summarized as follows:
\begin{itemize}
    \item We establish universal approximation results for a class of QNNs with explicit $\mathcal{L}^2$-error bounds.
    \item We derive complexity estimates showing that the required number of qubits scales logarithmically with the target accuracy.
    \item We propose two QPINN approaches, N-QPINN and A-QPINN, that enable quadrature-based and quadrature-free solution strategies, respectively, for nonlinear IDEs and FIPDEs.
    \item We validate the proposed methods through systematic numerical experiments on representative nonlinear IDEs and FIPDEs.
\end{itemize}
The remainder of the paper is organized as follows. Section~\ref{sec_problem-set} formulates the forward and inverse problems for IDEs and FIPDEs. Section~\ref{sec_bg} reviews classical PINNs and introduces the QNN architectures. Section~\ref{sec_approximation-error} provides a detailed analysis of the approximation capabilities of QNNs, including an explicit $\mathcal{L}_2$-error bound. Section~\ref{methodology} presents the proposed N-QPINN and A-QPINN methodologies. Section~\ref{num-exp} reports numerical experiments, and Section~\ref{sec_concl} concludes with remarks and potential future research directions.\\
\textbf{Notation:} We denote $C(\mathbb{R}^d)$ for the space of continuous real-valued functions on $\mathbb{R}^d$. For $1 \le p \le \infty$, the Lebesgue space $L^p(\mathbb{R}^d)$ consists of measurable functions with finite norm $\|f\|_{L^p}$. In particular, $\|f\|_{L^1}=\int_{\mathbb{R}^d}|f(x)|\,\mathrm{d}x$. For $f\in L^1(\mathbb{R}^d)$, the Fourier transform is defined by 
$\widehat f(\xi)=\int_{\mathbb{R}^d} f(x)e^{-i\,\xi\cdot x}\,\mathrm{d}x$, 
and the condition $\widehat f\in L^1(\mathbb{R}^d)$ means $\int_{\mathbb{R}^d}|\widehat f(\xi)|\,\mathrm{d}\xi<\infty$. We use $\|\cdot\|$ for the Euclidean norm in $\mathbb{R}^d$ and $\|\cdot\|_{L^2(\mu)}$ for the $L^2$ norm with respect to a probability measure $\mu$.
\section{Problem Statement}\label{sec_problem-set}
We consider a unified class of integro–differential and fractional integro–partial differential equations that arise in diverse applications, including fluid mechanics, heat transfer, viscoelasticity, anomalous transport, and multiscale physical systems. The following problem statements provide the general mathematical setting addressed by the proposed QPINN framework.
\subsection{Model Problem I: General Nonlinear IDE}
Let $\Omega\subset\mathbb{R}^d$ be a bounded domain with boundary $\partial\Omega$, and let $\boldsymbol{x}\in\Omega$. A nonlinear IDE can be written as
\begin{equation}\label{model-general-IDE}
\begin{cases}
\begin{aligned}
\mathscr{L}[u(\boldsymbol{x})]
&= f(\boldsymbol{x})
+ \lambda \int_{\mathcal{D}(\boldsymbol{x})}
\mathscr{K}(\boldsymbol{x},\boldsymbol{z})\,\mathcal{N}(u(\boldsymbol{z}))\, d\boldsymbol{z},
&& \boldsymbol{x}\in\Omega, \\[4pt]
\partial^{\boldsymbol{k}} u(\boldsymbol{x}_0)
&= a_{\boldsymbol{k}},
\quad |\boldsymbol{k}| < N,\; \boldsymbol{k}\in\mathbb{Z}_+^{d},
&& \text{(initial conditions)}, \\[4pt]
\mathscr{B}_j[u(\boldsymbol{x})]
&= g_j(\boldsymbol{x}),
\quad \boldsymbol{x}\in\partial\Omega,\; j=1,\dots,n,
&& \text{(boundary conditions).}
\end{aligned}
\end{cases}
\end{equation}
where $\mathscr{L}$ is a differential operator of order $N$,  
$\mathscr{K}$ is a kernel function, $\mathcal{D}(\boldsymbol{x})$ denotes a possibly $\boldsymbol{x}$-dependent integration region, and $\mathcal{N}$ is a nonlinear map acting on $u$. The operators $\mathscr{B}_j$, $j=1,\dots,n$, define boundary conditions prescribed on $\partial\Omega$.
\subsection{Model Problem II: General System of Nonlinear IDE}
Many applications involve multiple interacting fields, leading to coupled nonlocal systems of the form
\begin{equation}\label{model-general-system}
\begin{cases}
\displaystyle
\mathscr{L}_i\bigl[u_1,\dots,u_m\bigr](\boldsymbol{x})=
f_i(\boldsymbol{x})
+ \lambda_i \displaystyle\int_{\mathcal{D}(\boldsymbol{x})}
\mathscr{K}_i(\boldsymbol{x},\boldsymbol{z})\,
\mathcal{N}_i\!\left(u_1(\boldsymbol{z}),\dots,u_m(\boldsymbol{z})\right)\, d\boldsymbol{z},\\
\hfill \boldsymbol{x}\in\Omega,\; i=1,\dots,m,\\
\displaystyle
\partial^{\boldsymbol{k}} u_i(\boldsymbol{x}_0)
= a_{i,\boldsymbol{k}}, 
\qquad |\boldsymbol{k}|<N_i,\;\boldsymbol{k}\in\mathbb{Z}_+^d,
\;\;\text{(initial conditions)}, \\[6pt]
\displaystyle
\mathscr{B}_{i,j}\bigl[u_i(\boldsymbol{x})\bigr]
= g_{i,j}(\boldsymbol{x}), 
\qquad \boldsymbol{x}\in\partial\Omega,\;\; j=1,\dots,n_i,
\hfill \text{(boundary conditions).}\\
\end{cases}
\end{equation}
where $\mathscr{L}_i$ are differential operators associated with each component $u_i$, 
$\mathscr{K}_i$ are kernel functions, and $\mathcal{N}_i$ denote nonlinear maps coupling the fields $u_1,\dots,u_m$. The operators $\mathscr{B}_{i,j}$, $j=1,\dots,n_i$, define the boundary conditions for the $i$-th component on $\partial\Omega$.
\subsection{Model Problem III: General Nonlinear Fractional Integro PDE}
To capture memory effects and anomalous diffusion, we consider equations involving Caputo fractional derivatives and fractional integrals:
\begin{equation}\label{model-general-FIPDE}
\begin{cases}
\displaystyle
\mathscr{L}\bigl[u(\boldsymbol{x},t)\bigr]
+ \sum_{j=1}^{J} D_t^{\alpha_j} u(\boldsymbol{x},t)
+ \sum_{k=1}^{K} \mathcal{I}_t^{\beta_k} u(\boldsymbol{x},t) \\[4pt]
\displaystyle
\qquad\qquad\qquad\qquad= f(\boldsymbol{x},t)+ \lambda \displaystyle\int_{\mathcal{D}(\boldsymbol{x},t)}
\mathscr{K}\bigl((\boldsymbol{x},t),(\boldsymbol{z},s)\bigr)\,
\mathcal{N}\!\left(u(\boldsymbol{z},s)\right)\, d\boldsymbol{z}\,ds, \\
\qquad\qquad\qquad\qquad\qquad\qquad\qquad\qquad\qquad\qquad\qquad\quad(\boldsymbol{x},t)\in\Omega\times(0,T), \\
\displaystyle 
u(\boldsymbol{x},0)=u_0(\boldsymbol{x}), 
\qquad \partial_t u(\boldsymbol{x},0)=u_1(\boldsymbol{x}),
\qquad \boldsymbol{x}\in\Omega,\\[10pt]
\displaystyle 
\mathscr{B}_j\bigl[u(\boldsymbol{x},t)\bigr]=g_j(\boldsymbol{x},t),
\; j=1,\dots,n, \quad \boldsymbol{x}\in\partial\Omega,
\quad\text{(boundary conditions).}
\end{cases}
\end{equation}
where $\mathscr{L}$ denotes a spatial differential operator acting on $u(\boldsymbol{x},t)$, 
$D_t^{\alpha_j}$ represents the Caputo fractional derivative of order $\alpha_j\in(0,1)$, 
and $\mathcal{I}_t^{\beta_k}$ denotes the fractional integral operator of order $\beta_k\in(0,1)$. 
The function $\mathscr{K}$ is a kernel describing the nonlocal interaction over the space–time region 
$\mathcal{D}(\boldsymbol{x},t)$, while $\mathcal{N}$ is a nonlinear map acting on the solution $u$. 
The operators $\mathscr{B}_j$, $j=1,\dots,n$, define the boundary conditions on $\partial\Omega$. 
The Caputo fractional derivative of order $\alpha$ and the fractional integral of order $\beta$ are defined as
\begin{equation}\label{def-capto}
\begin{aligned}
D_t^{\alpha} u(\boldsymbol{x},t)
&= \frac{1}{\Gamma(1-\alpha)}
\int_{0}^{t} \frac{\partial u(\boldsymbol{x},s)}{\partial s}(t-s)^{-\alpha}\, ds, \\[6pt]
\mathcal{I}_t^{\beta} u(\boldsymbol{x},t)
&= \frac{1}{\Gamma(\beta)}
\int_{0}^{t} (t-s)^{\beta-1} u(\boldsymbol{x},s)\, ds.
\end{aligned}
\end{equation}
These models combine local differential operators with nonlocal integral and fractional operators, leading to solution behaviors that are difficult to approximate accurately. The objective of this work is to develop QPINN formulations capable of solving the general models \eqref{model-general-IDE}, \eqref{model-general-system}, and \eqref{model-general-FIPDE}. To this end, we introduce two complementary methodologies, N-QPINN and A-QPINN, each tailored to different types of nonlocal structures. The A-QPINN formulation is based on auxiliary functions and avoids numerical quadrature by reformulating integral operators into equivalent differential systems; however, it is not well suited for fractional operators, which have inherently nonlocal, weakly singular kernel structures. In contrast, the N-QPINN approach employs high-order numerical quadrature to accurately handle such nonlocal and fractional terms. We demonstrate the effectiveness of these approaches on representative benchmark problems presented in the numerical section. In the next section, we review the necessary background on classical PINNs and introduce the quantum neural network architectures underlying the proposed QPINN framework.


\section{Background and Related Works}\label{sec_bg}
In this section, we first review the classical PINNs framework~\cite{raissi2019physics} for solving PDEs and summarize existing extensions designed for integro-differential and fractional PDEs. We then provide a brief introduction to QNNs and outline the QPINN approach that forms the foundation of our proposed methodology.
\subsection{Physics Informed Neural Network (PINN)}
Physics-informed neural networks (PINNs) are a class of deep learning models that incorporate the governing physical laws of a system into the training process of a neural network. They have been widely used for solving differential equations, especially PDEs, without relying solely on traditional discretization schemes. To illustrate the core idea, consider a general PDE of the form:
\begin{equation}\label{gen-pde}
\mathscr{L}\left[u(\boldsymbol {x})\right] = f(\boldsymbol {x}), \quad \boldsymbol {x} \in \Omega \subset \mathbb{R}^d,
\end{equation}
where $u(\boldsymbol {x})$ is the unknown solution, $\mathscr{L}\left[\cdot\right]$ denotes a differential operator, and $\Omega$ is the spatial domain. The PDE is typically subject to the boundary conditions $\mathcal{B}\left[u(\boldsymbol {x})\right]= b(\boldsymbol {x}), \quad \boldsymbol {x} \in \partial\Omega$. In PINNs, a neural network $u_{\theta}(\boldsymbol {x})$, parameterized by weights $\theta$, is trained to approximate $u(\boldsymbol {x})$. The network output is differentiated with respect to $\boldsymbol {x}$ using automatic differentiation to evaluate the PDE residual, defined as:
\begin{equation*}
\mathcal{R}_{\theta}(\boldsymbol {x}) = \mathscr{L}\left[u_{\theta}(\boldsymbol {x})\right] - f(\boldsymbol {x}).
\end{equation*}
The total loss function is then defined as the combination of two contributions:
\begin{equation}
\mathcal{L}_{\text{Total}} = \mathcal{L}_{PDE} + \mathcal{L}_{BC},
\end{equation}
where
\begin{equation*}
\left\{
\begin{aligned}
\mathcal{L}_{PDE} &= \frac{1}{N_{r}} \sum_{i=1}^{N_{r}} \Big| \hat{\mathcal{R}}(\boldsymbol {x}_{r}^{i}) \Big|^{2}, \\[6pt]
\mathcal{L}_{BC} &= \frac{1}{N_{BC}} \sum_{i=1}^{N_{BC}} \Big| \mathcal{B}[u_\theta](\boldsymbol {x}_{b}^{i}) - g(\boldsymbol {x}_{b}^{i}) \Big|^{2}.
\end{aligned}
\right.
\end{equation*}
Here, $\{\boldsymbol {x}_{r}^{i}\}_{i=1}^{N_r}$ are the collocation points for enforcing the PDE, and $\{\boldsymbol {x}_{b}^{i}\}_{i=1}^{N_{BC}}$ are the boundary condition points. Minimizing this loss ensures that the neural network approximation $u_{\theta}(\boldsymbol {x})$ adheres to both the governing equation and the prescribed conditions. Although PINNs were initially developed to solve PDEs, recent advances have extended their applicability to IDEs and FPIDEs \cite{Pang2019fPINNs, Lulu-DeepXDE2021, Yuan2022APINN}. The principal difficulty in these extended settings arises from the evaluation of integral and fractional operators, which cannot be directly computed via automatic differentiation. A commonly adopted approach to address this issue is the fractional PINN (fPINN) framework \cite{Pang2019fPINNs}, wherein integral and fractional derivatives are discretized through appropriate numerical schemes, effectively transforming the original equation into a PDE form amenable to standard PINN architectures. To circumvent the need for discretizing integral operators, the Auxiliary PINN (A-PINN) framework \cite{Yuan2022APINN} was proposed for IDEs, wherein auxiliary variables are introduced to represent integral terms explicitly. However, this strategy cannot be straightforwardly extended to fractional PDEs (fPDEs) or fIPDEs, as fractional derivatives possess inherent nonlocality and do not admit a closed-form differential reformulation in terms of auxiliary variables. In the following, we briefly review existing quantum PINN formulations for solving PDEs, which provide the conceptual basis for our proposed quantum-augmented frameworks.
\subsection{Quantum Physics-Informed Neural Network (Q-PINN)}
Quantum Physics-Informed Neural Network (Q-PINN) methodologies have recently emerged as promising alternatives to classical PINNs for solving PDEs. The central idea in Q-PINNs is to replace the classical neural network surrogate with a quantum neural network (QNN), while retaining the overall optimization and training paradigm of the standard PINN methodology. The use of quantum circuits to represent PDE trial solutions within a physics-informed framework was first introduced by Kyriienko, Paine, and Elfving \cite{Kyriienko2021}. In this formulation, the space–time coordinates are initially embedded into a quantum state through a unitary feature-map operator, which encodes the input variables into the amplitudes or phases of a set of qubits. The resulting state is then evolved by a parameterized quantum circuit (PQC) composed of variational layers and entangling gates, whose tunable parameters are optimized to minimize the physics-informed loss. Finally, the network output is obtained by evaluating the expectation value of an appropriately chosen Hermitian observable, which serves as the readout operator defining the scalar or vector-valued prediction of the Q-PINN model.
\begin{figure}[ht]
    \centering
    \includegraphics[width=0.9\linewidth]{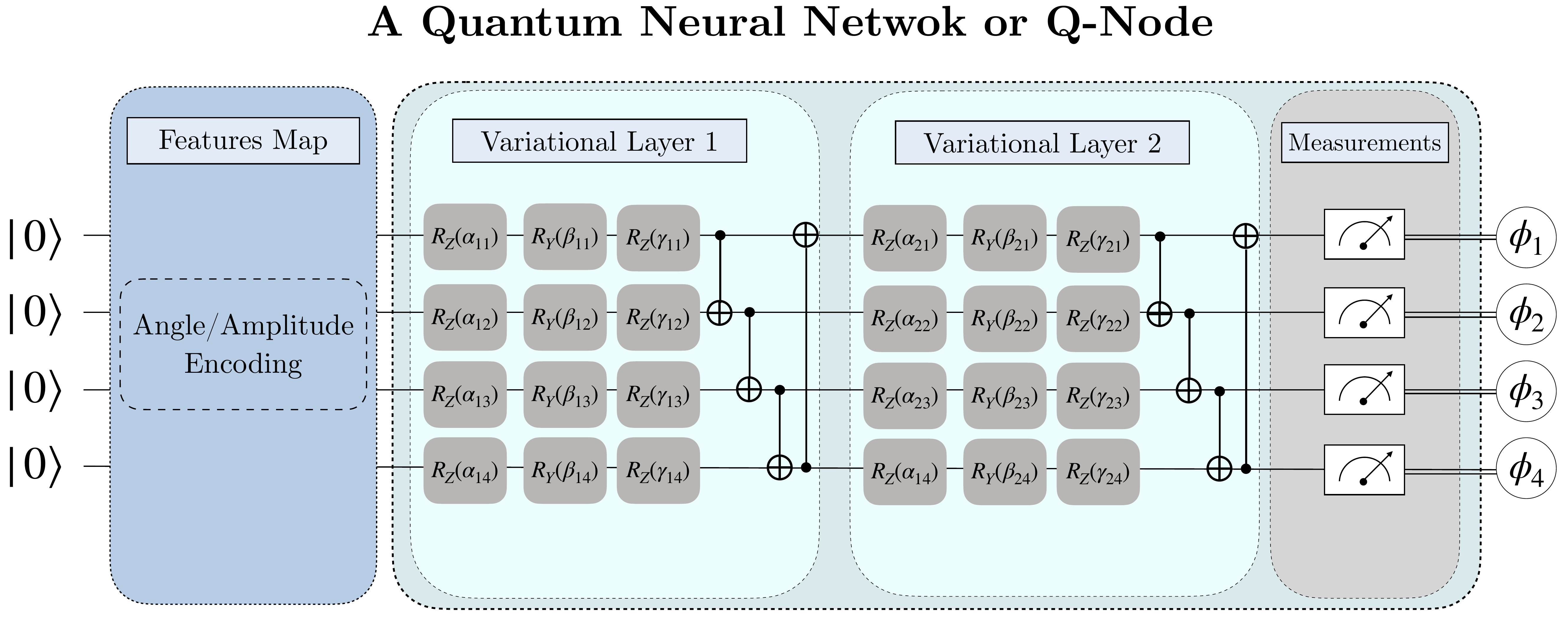}
    \caption{General QNN Architecture: Flow from feature map encoding, through two variational layers, to output measurements on a 4-qubit register, representative of general multi-layer QNN architectures employed in quantum machine learning models.}
    \label{fig:qnn_architecture}
\end{figure}
\subsubsection{Construction of QNN}\label{const-qnn}
The general architecture of a QNN node for the Q-PINN framework is illustrated in Figure~\ref{fig:qnn_architecture}. The QNN is composed of several key modules. The first step is the feature map, where input variables are mapped to a quantum state using a unitary feature map. For given input $\boldsymbol{x}=(x_1,x_2,\hdots,x_d)\in\mathbb{R}^{d}$ and the number of qubits $n$, let the trainable parameters be $\boldsymbol{a}=(\boldsymbol{a}_1,\boldsymbol{a}_2,\ldots,\boldsymbol{a}_n)\in (\mathbb{R}^d)^n$ and $\boldsymbol{b}=(b_1,\ldots,b_n)\in\mathbb{R}^n$. These are first mapped to an $n$-dimensional latent vector through an affine transformation
\begin{equation*}
    \boldsymbol{z} = \mathcal{W}\boldsymbol{x} + \boldsymbol{b},\qquad\text{ where }\quad 
    \mathcal{W} = \begin{pmatrix}
        \boldsymbol{a}_{1}^{T} \\
        \boldsymbol{a}_{2}^{T} \\
        \vdots \\
        \boldsymbol{a}_{n}^{T} \\
    \end{pmatrix}\in \mathbb{R}^{n\times d}, \ \boldsymbol{b}\in\mathbb{R}^{n},
\end{equation*}
so that the $i$-th component is $z_i =  \boldsymbol{a}_i^{T}\boldsymbol{x} + b_i,$ for $i = 1, 2, \dots, n.$
Then we use rotation gates to encode each component $z_i$ into the $i$-th qubit. In our implementation, we encode inputs using single-qubit rotations about the $y$-axis, which is given as
\begin{equation*}
    R_y(z_i)=\begin{pmatrix}
\cos(z_i/2) & -\sin(z_i/2)\\[4pt]
\sin(z_i/2) & \cos(z_i/2)
\end{pmatrix},\qquad i=1,\dots,n,
\end{equation*}
Therefore, the collective operation of encoding all components of $\boldsymbol{z}$ on the $n$ qubits is given by
\begin{equation}\label{eq-feature-map}
U_{\psi(\boldsymbol{x})}=\bigotimes_{i=1}^n R_y(z_i).
\end{equation}
Equivalently, one may apply a Hadamard before an $R_z$ encoding; both constructions create a nontrivial superposition and are mathematically equivalent up to parameter reparametrization. After feature encoding, the qubits evolve under $L$ layers of a parametrized quantum circuit. Each layer consists of arbitrary single-qubit rotation gates of the form
\begin{equation*}
R(\alpha_{\ell i},\beta_{\ell i},\gamma_{\ell i}) = R_z(\alpha_{\ell i})\,R_y(\beta_{\ell i})\,R_z(\gamma_{\ell i}),
\qquad \text{where }
R_z(\alpha)=
\begin{pmatrix}
e^{-i {\alpha}/{2}} & 0 \\
0 & e^{i {\alpha}/{2}}
\end{pmatrix},
\end{equation*}
followed by an entangling operation $E$ composed of a chain of CNOT gates arranged in a circular pattern:
\begin{equation*}
    E = \left(\prod_{k=1}^{n-1} \mathrm{CNOT}_{k,k+1}\right) \mathrm{CNOT}_{n,1}.
\end{equation*}
Thus, the $\ell$-th variational layer is written as
\begin{equation*}
    U(\Gamma^{\ell})= 
 E\left(\bigotimes_{i=1}^{n} 
R_z(\gamma_{\ell i})
R_y(\beta_{\ell i})
R_z(\alpha_{\ell i})
\right),
\end{equation*}
with parameters $\Gamma^\ell = \left\{(\alpha_{\ell i},\beta_{\ell i},\gamma_{\ell i}) : i=1,\dots,n\right\}$. For $L$ layers we define the parameterized unitary operator by composing the layers in order
\begin{equation*}
    U_{\Theta}(\Gamma) = \prod_{\ell=1}^{L} U(\Gamma^\ell),
    \qquad \text{where } \Theta = \left\{\Gamma^1, \Gamma^2, \dots, \Gamma^L\right\}.
\end{equation*}
Finally, the measurement layer of the network yields the output by computing the expectation value of a Hermitian operator. In this context, we select the Pauli-$Z$ operator as the observable for each qubit. For the $i$-th qubit, the output is
\begin{equation}\label{eq-ith-layerop}
    \phi_i{(\boldsymbol{x})}=\bra{0}^{\otimes n}\left(U_{\Theta}(\Gamma)U_{\psi(\boldsymbol{x})}\right)^{\dagger}Z_i  \left(U_{\Theta}(\Gamma)U_{\psi(\boldsymbol{x})}\right)\ket{0}^{\otimes n}, \qquad Z_i = I^{\otimes(i-1)}\otimes Z\otimes I^{\otimes (n-i)},
\end{equation}
which produces the $n$-dimensional vector
\begin{equation*}
    \Phi_{\Theta}(\boldsymbol{x}) = \left(\phi_1{(\boldsymbol{x})}, \phi_2{(\boldsymbol{x})}, \hdots, \phi_n{(\boldsymbol{x})}\right).
\end{equation*}
A classical linear layer maps these quantum expectations to the network prediction,
\begin{equation}\label{qnn-op}
    u_{\theta}{(\boldsymbol{x})} = \boldsymbol{c}^{T}\Phi_{\Theta}(\boldsymbol{x})+c_0,
\end{equation}
where $\boldsymbol{c}\in\mathbb{R}^n$ and $c_0\in\mathbb{R}$ are classical weights, and $\theta = \left\{\mathcal{W}, \boldsymbol{b},\Gamma, \boldsymbol{c}, c_0\right\}$ is the full set of parameters.

\subsubsection{QPINN Approach for Solving PDEs}
To understand the QPINN approach for solving differential equations, we consider the same generic governing equation~\eqref{gen-pde} used in the classical PINN formulation. In the QPINN framework, the classical neural network surrogate is replaced by the QNN constructed in Section~\ref{const-qnn}. The QNN output $u_{\theta}(\boldsymbol{x})$, computed from expectation values of the parameterized quantum circuit described in~\eqref{qnn-op}, serves as the predicted solution for the PDE. Using this quantum-enhanced trial solution, the PDE residual at an interior collocation point $\boldsymbol{x}_{r}^{i}$ is evaluated as
\begin{equation}
\mathcal{R}_{\theta}(\boldsymbol{x}_{r}^{i})
=
\mathscr{L}\!\left[u_{\theta}(\boldsymbol{x}_{r}^{i})\right]
-
f(\boldsymbol{x}_{r}^{i}).
\end{equation}
The total QPINN loss is defined analogously to the classical PINN loss:
\begin{equation}
\mathcal{L}_{\mathrm{Total}}(\theta)
=
\mathcal{L}_{\mathrm{PDE}}(\theta)
+
\mathcal{L}_{\mathrm{BC}}(\theta),
\end{equation}
where
\begin{equation}
\left\{
\begin{aligned}
\mathcal{L}_{\mathrm{PDE}}(\theta)
&=
\frac{1}{N_{r}}
\sum_{i=1}^{N_{r}}
\left|
\mathcal{R}_{\theta}(\boldsymbol{x}_{r}^{i})
\right|^{2}, \\[6pt]
\mathcal{L}_{\mathrm{BC}}(\theta)
&=
\frac{1}{N_{BC}}
\sum_{i=1}^{N_{BC}}
\left|
\mathcal{B}[u_{\theta}](\boldsymbol{x}_{b}^{i})
-
g(\boldsymbol{x}_{b}^{i})
\right|^{2}.
\end{aligned}
\right.
\end{equation}
Each evaluation of this loss requires executing the quantum circuit to obtain $\Phi_{\Theta}(\boldsymbol{x})$, differentiating the resulting QNN output with respect to the spatial variables, and forming the corresponding PDE and boundary residuals. The optimization of QPINN parameters follows a hybrid quantum--classical workflow, wherein a classical optimizer such as Adam or L-BFGS updates the complete parameter set $\theta = \{\mathcal{W}, \boldsymbol{b}, \Gamma, \boldsymbol{c}, c_{0}\}$, with $\Gamma$ denoting the variational quantum parameters. During each training iteration, the quantum circuit is evaluated at all collocation and boundary points to compute the trial solution; the PDE and boundary residuals are assembled to form the total loss; gradients with respect to quantum parameters are obtained using the parameter-shift rule, while gradients with respect to classical parameters are computed via standard automatic differentiation; and the parameters are updated to decrease the physics-informed loss.

\section{Approximation Capabilities of the QNN}\label{sec_approximation-error}
In this section we investigate the approximation capabilities of the proposed QNN architecture. We begin by characterizing the exact functional form of the QNN output, showing that each quantum expectation yields a real trigonometric polynomial in the affine encoded inputs. Building on this structure, we construct cosine ridge functions as basic primitives and subsequently establish quantitative $L^2$ universal approximation results, demonstrating the expressive power of the model.
\begin{proposition}\label{prop-1}
Let $n\in\mathbb{N}$ be the number of qubits, and let 
$\theta=\{\boldsymbol{a},\boldsymbol{b},\Gamma,\boldsymbol{c},c_0\}$ 
denote the full set of QNN parameters.
Then, for any input $\boldsymbol{x}\in\mathbb{R}^d$, the QNN output takes the form 
\begin{equation}
    u_{\theta}(\boldsymbol{x})
    =
    \sum_{i=1}^{n} 
    c_i\,F_i\!\big(\boldsymbol{a}_1^{T}\boldsymbol{x}+b_1,\dots,
                   \boldsymbol{a}_n^{T}\boldsymbol{x}+b_n;\Gamma\big)
    + c_0,
\end{equation}
where each $F_i(\cdot\,;\Gamma)$ is a real trigonometric polynomial depending only on the variational parameters $\Gamma$ and not on the input $\boldsymbol{x}$.
\end{proposition}
\begin{proof} Recall from the section~\ref{const-qnn}  the QNN forward unitary is
\[
U(\boldsymbol{x};\Theta) \;=\; U_{\Theta}(\Gamma)\,U_{\psi(\boldsymbol{x})},
\]
where the $U_{\psi(\boldsymbol{x})}$ is the feature map as defined in equation~\eqref{eq-feature-map} and $U_{\Theta}(\Gamma)$ is the variational unitary depending only on
the variational parameters $\Gamma$. 
For the $i$-th qubit, the output is the expectation value with respect to the measurement operator the Pauli-$Z_i$ is written as 
\begin{equation}
    \phi_i(\boldsymbol{x})
= \langle 0|^{\otimes n}\,U(\boldsymbol{x};\Theta)^\dagger\,Z_i\,
U(\boldsymbol{x};\Theta)\,|0\rangle^{\otimes n}.
\end{equation}
Define the operator
\[
M_i(\Gamma) \;:=\; U_{\Theta}(\Gamma)^\dagger\,Z_i\,U_{\Theta}(\Gamma).
\]
By construction $M_i(\Gamma)$ depends only on $\Gamma$ not on $\boldsymbol{x}$.
Therefore, we can write 
\begin{equation}\label{op-qubiti}
    \phi_i(\boldsymbol{x})
= \langle 0|^{\otimes n}\,W(\boldsymbol{z})^\dagger
\;M_i(\Gamma)\; W(\boldsymbol{z})\,|0\rangle^{\otimes n},
\qquad
W(\boldsymbol{z}):=\bigotimes_{j=1}^n R_y(z_j).
\end{equation}
The Pauli operators form a basis for the space of $2^n\times 2^n$
Hermitian matrices. Hence we can expand $M_i(\Gamma)$ as a linear
combination of Pauli strings:
\begin{equation}\label{eq:pauli_expansion}
\begin{aligned}
M_i(\Gamma)
&= \sum_{j} c_{i,j}(\Gamma)\, \mathcal{P}_j, \\
\mathscr{P}_n
&:= \left\{\, \mathcal{P}_j = \bigotimes_{s=1}^n P_s :
P_s \in \{\alpha_{s,0}I+\alpha_{s,1}X+\alpha_{s,2}Y+\alpha_{s,3}Z :
\alpha_{s,i}\in\mathbb{R}\}\,\right\}.
\end{aligned}
\end{equation}
where the coefficients $c_{i,j}(\Gamma)$ are function depenting on the trainable parameters $\Gamma$. We now conjugate each Pauli string $\mathcal{P}_j \in \mathscr{P}_n $ by the encoding unitary $W(\boldsymbol{z})$.

\begin{align*}
W(\boldsymbol{z})^\dagger \mathcal{P}_j W(\boldsymbol{z})
&= \bigotimes_{s=1}^n R_y(z_s)^\dagger P_s R_y(z_s) \\
&= \bigotimes_{s=1}^n R_y(z_s)^\dagger \left(\alpha_{s,0}I+\alpha_{s,1}X+\alpha_{s,2}Y+\alpha_{s,3}Z\right) R_y(z_s),\quad\alpha_{s,i}\in \mathbb{R} \\
&= \bigotimes_{s=1}^n \Big(\alpha_{s,0}I+(\alpha_{s,1}\cos z_s-\alpha_{s,3}\sin z_s)X \\
&\qquad\qquad\qquad+\alpha_{s,2}Y+ (\alpha_{s,1}\sin z_s+\alpha_{s,3}\cos z_s) Z\Big)\\
&= \sum_{P_s \in \{I,X,Y,Z\}} \left( \prod_{s=1}^n \lambda_{s,P_s}(z_s) \right)\bigotimes_{s=1}^n P_s = \sum_{\mathcal{P}_\alpha \in \mathscr{P}_n}\Lambda_{j,\alpha}(\boldsymbol{z})\mathcal{P}_\alpha,
\end{align*}
where each $\lambda_{s,P_s}(z_s)$ is either a constant or a linear combination of 
$\cos{z_{s}}$ or $\sin{z_{s}}$ and $\Lambda_{j,\alpha}(\boldsymbol{z})$  is the product of $\lambda_{s,P_s}(z_s)$'s. Therefore $\Lambda_{j,\alpha}(\boldsymbol{z})$ is a real trigonometric polynomial in the variable $\boldsymbol{z}$.
Substituting this expansion into \eqref{op-qubiti} gives
\begin{align*}
\phi_i(\boldsymbol{x})
&= \langle 0|^{\otimes n}W(\boldsymbol{z})^\dagger
M_i(\Gamma)W(\boldsymbol{z})|0\rangle^{\otimes n}=\langle 0|^{\otimes n}W(\boldsymbol{z})^\dagger
\left(\sum_{j} c_{i,j}(\Gamma)\mathcal{P}_j\right)W(\boldsymbol{z})\,|0\rangle^{\otimes n}\\
=&\langle 0|^{\otimes n}\sum_j c_{i,j}(\Gamma)\,
W(\boldsymbol{z})^\dagger \mathcal{P}_j W(\boldsymbol{z})
|0\rangle^{\otimes n}
= \langle 0|^{\otimes n}\sum_j c_{i,j}(\Gamma)
\sum_{\mathcal{P}_\alpha\in\mathscr{P}_n}
\Lambda_{j,\alpha}(\boldsymbol{z}) \mathcal{P}_\alpha
|0\rangle^{\otimes n}\\
=&\sum_{j} c_{i,j}(\Gamma) \sum_{\mathcal{P}_\alpha \in \mathscr{P}_n} \Lambda_{j,\alpha}(\boldsymbol{z}) \langle 0|^{\otimes n} \mathcal{P}_\alpha |0\rangle^{\otimes n}.
\end{align*}
The expectation $\langle 0|^{\otimes n} \mathcal{P}_\alpha |0\rangle^{\otimes n}$ vanishes unless 
$\mathcal{P}_\alpha\in \mathscr{P}_n^{I,Z}\subset \mathscr{P}_n$ the set of Pauli strings composed exclusively of $I$ and $Z$. For such strings, each qubit contributes a factor $\langle 0| I |0\rangle = 1$ or $\langle 0| Z |0\rangle = 1$; 
hence the total expectation equals $1$. Then
\begin{equation*}
    \phi_i(\boldsymbol{x})
= \sum_{j} c_{i,j}(\Gamma) 
\sum_{\mathcal{P}_\alpha \in \mathscr{P}_n^{I,Z}} \Lambda_{j,\alpha}(\boldsymbol{z}).
\end{equation*}
Every $\Lambda_{j,\alpha}(\boldsymbol{z})$ is, by construction, a product of factors 
$\lambda_{s,P_s}(z_s)$ with $P_s \in \{I,Z\}$. 
Since each $\lambda_{s,I}(z_s)=\alpha_{s,0}$ is constant and each 
$\lambda_{s,Z}(z_s)$ is a linear trigonometric monomial, their product is a real trigonometric polynomial in the variables $z_1,\dots,z_n$. Consequently, $\phi_i(\boldsymbol{x})$ itself is a trigonometric polynomial 
\begin{equation*}
    \phi_i(\boldsymbol{x}) = F_i\big(z_1,\dots,z_n;\Gamma\big),
\end{equation*}
where the coefficients of $F_i$ depend only on the variational parameters $\Gamma$. Therefore, $\phi_i(\boldsymbol{x})$ can be written as
\begin{equation*}
    \phi_i(\boldsymbol{x}) = F_i\!\big(\boldsymbol{a}_1^{T}\boldsymbol{x}+b_1,\dots,
                   \boldsymbol{a}_n^{T}\boldsymbol{x}+b_n;\Gamma\big).
\end{equation*}
Finally, the QNN output $u_{\theta}(\boldsymbol{x})$ is a weighted sum of these expectation values:
\begin{equation*}
    u_{\theta}(\boldsymbol{x}) 
= \sum_{i=1}^{n} c_i\,\phi_i(\boldsymbol{x}) + c_0
= \sum_{i=1}^{n} c_i\,F_i\!\big(\boldsymbol{a}_1^{T}\boldsymbol{x}+b_1,\dots,
                   \boldsymbol{a}_n^{T}\boldsymbol{x}+b_n;\Gamma\big) + c_0,
\end{equation*}
which is exactly the form claimed in the proposition.
\end{proof}
From the product–trigonometric representation of each $F_i$ obtained in the proof of Proposition~\ref{prop-1}, we may further rewrite every monomial in the variables $\{\cos z_j,\sin z_j\}$ using standard product-to-sum identities.  This yields an equivalent finite Fourier-series representation in terms of cosine functions of integer linear combinations of the angles.  The resulting form is stated in the following corollary.

\begin{corollary}\label{coroll-1}
Under the same assumptions as Proposition~\ref{prop-1}, let $z_j=\boldsymbol{a}_j^{T}\boldsymbol{x}+b_j$.
Then for each $i=1,\dots,n$ there exists a finite index set 
$\mathcal{K}_i\subset\mathbb{Z}^n$ and real coefficients
$A_{\boldsymbol{k},i}(\Gamma)$ and phases $\delta_{\boldsymbol{k},i}(\Gamma)$,
depending only on $\Gamma$, such that
\[
F_i(z_1,\dots,z_n;\Gamma)
= \sum_{\boldsymbol{k}\in\mathcal{K}_i} 
    A_{\boldsymbol{k},i}(\Gamma)\,\cos\!\big(\boldsymbol{k}\cdot\boldsymbol{z}+\delta_{\boldsymbol{k},i}(\Gamma)\big).
\]
Equivalently, each $F_i$ is a real polynomial in the one-variable trigonometric
functions $\{\cos z_j,\sin z_j\}_{j=1}^n$ whose coefficients depend only on $\Gamma$.
\end{corollary}
\begin{proof}
By Proposition~\ref{prop-1}, each $F_i$ is a finite real polynomial in the variables
$\{\cos z_j,\sin z_j\}$; explicitly,
\begin{equation}
F_i(z)
= \sum_{\ell,m} c_{\ell,m}(\Gamma)\prod_{j=1}^n (\cos z_j)^{\ell_j}(\sin z_j)^{m_j},
\end{equation}
where $\ell=(\ell_1,\dots,\ell_n), m=(m_1,\dots,m_n)$ range over a finite set of multi-indices and the coefficients $c_{\ell,m}(\Gamma)\in\mathbb{R}$ depend only on $\Gamma$. Fix a single monomial
\begin{equation}
    M_{\ell,m}(z)=\prod_{j=1}^n(\cos z_j)^{\ell_j}(\sin z_j)^{m_j}.
\end{equation}
For each coordinate $j$ one may expand $(\cos z_j)^{\ell_j}(\sin z_j)^{m_j}$
into a finite sum of one-variable cosines with integer multiples of $z_j$ as arguments
(e.g.\ by repeatedly applying product-to-sum identities or by using the complex exponential representation).
Thus there exist $T_j\in\mathbb{N}$ and real constants $\alpha_{j,t}$, $\delta_{j,t}$
such that
\[
(\cos z_j)^{\ell_j}(\sin z_j)^{m_j}
= \sum_{t=-T_j}^{T_j}\alpha_{j,t}\cos(t z_j+\delta_{j,t}).
\]
Multiplying these finite sums over $j=1,\dots,n$ and expanding yields a finite sum
of products of one-variable cosines $\prod_{j=1}^n\cos(t_j z_j+\delta_{j,t_j})$.
Each such product can in turn be converted, by repeated use of
$\cos A\cos B=\tfrac12(\cos(A+B)+\cos(A-B))$, into a finite linear combination of
cosines of integer linear forms $\boldsymbol{k}\cdot\boldsymbol{z}$ with phases.
Collecting terms across the finitely many monomials $(\ell,m)$ produces the stated finite cosine expansion,
with $\mathcal{K}_i$ finite and all coefficients/phases depending only on $\Gamma$.
\end{proof}
Next, we show that the QNN architecture can realize elementary cosine ridge functions, which serve as the basic building blocks of classical Fourier feature approximations. The corollary shows that each component $F_i$ may be represented as a finite Fourier series; in the sequel, we will use the following constructive lemma that exhibits how a single cosine ridge function is realized by one qubit.
\begin{lemma}\label{lemma-1}
Let $\boldsymbol{\omega}\in\mathbb{R}^d$ and $\tau\in\mathbb{R}$. Then there exist QNN parameters $\theta=\{\boldsymbol{a},\boldsymbol{b},\Gamma,\boldsymbol{c},c_0\}$ such that for some qubit index $i$ the corresponding quantum expectation
satisfies
\begin{equation}
    \phi_i(\boldsymbol{x})= K\,\cos(\boldsymbol{\omega}^\top\boldsymbol{x}+\tau)
\qquad\text{for all }\boldsymbol{x}\in\mathbb{R}^d,
\end{equation}
for any amplitude $K$ with $0\le K\le 1$. Moreover, any amplitude $K'>0$ can be realized in the full QNN output by taking a unit-amplitude quantum expectation and setting the classical readout weight $c_i=K'$.
\end{lemma}

\begin{proof}
Choose parameters so that only qubit $i$ depends on the input: set $\boldsymbol{a}_j=\mathbf{0}, b_j=0$ for all $j\neq i$, and $\boldsymbol{a}_i=\boldsymbol{\omega}, b_i=0$, so $z_i=\boldsymbol{\omega}^\top\boldsymbol{x}$ and $z_j=0$ for $j\neq i$. Arrange the variational parameters so that the whole circuit factors as
$U(\boldsymbol{x};\Theta)=U_i(\boldsymbol{x})\otimes I_{\mathrm{rest}}$; hence only the $i$-th qubit contributes a nontrivial expectation.

Let the single-qubit variational block on qubit $i$ be the Euler decomposition
$U_i=R_z(\alpha)R_y(\beta)R_z(\gamma)$, and define $M:=U_i^\dagger Z U_i$.
As a traceless Hermitian $2\times2$ matrix there exists a real unit vector $r=(r_x,r_y,r_z)$ with
$M=r_x X + r_y Y + r_z Z$. For the encoded state $|\psi(z_i)\rangle=R_y(z_i)|0\rangle$ a direct calculation yields
\begin{equation}
\phi_i(\boldsymbol{x})=\langle\psi(z_i)|M|\psi(z_i)\rangle
= r_z\cos z_i + r_x\sin z_i,
\qquad z_i=\boldsymbol{\omega}^\top\boldsymbol{x}.
\end{equation}
Therefore, the quantum expectation is a sinusoidal combination with an amplitude less than 1. Choosing $r_z = K\cos\tau,\; r_x = -K\sin\tau$ is possible exactly when $0\le K\le 1$, which gives
\begin{equation}
    \phi_i(\boldsymbol{x}) = K\cos(z_i+\tau)=K\cos(\boldsymbol{\omega}^\top\boldsymbol{x}+\tau).
\end{equation}
Finally, if an amplitude $K'>1$ is desired for the network output, realize a unit-amplitude cosine on the quantum side (choose $K=1$) and set the classical readout weight $c_i=K'$, so that the full QNN realizes $c_i\phi_i(\boldsymbol{x})=K'\cos(\cdot)$.
\end{proof}
\begin{theorem}\label{thm:l2_uat_concise}
Let $\mu$ be a probability measure on $\mathbb{R}^d$.
Suppose $f \in C(\mathbb{R}^d)\cap L^1(\mathbb{R}^d)$ and its Fourier transform 
$\widehat{f}\in L^1(\mathbb{R}^d)$.
For every integer $n\ge 1$ there exist trainable parameters $\theta=\{\boldsymbol{a},\boldsymbol{b},\Gamma,\boldsymbol{c},c_0\},$ with $n$ qubits, such that the QNN output $u_\theta(\boldsymbol{x})$ satisfies
\begin{equation*}
    \|f - u_{\theta}\|_{L^2(\mu)}
    \le
    \frac{\|\widehat{f}\|_{L^1}}{\sqrt{n}}.
\end{equation*}
\end{theorem}
\begin{proof}
Since $f\in L^{1}(\mathbb{R}^{d})$ and $\widehat{f}\in L^{1}(\mathbb{R}^{d})$, the classical Fourier inversion theorem guarantees that $f$ admits the pointwise representation
\begin{equation*}
    f(x)
    = \Re \int_{\mathbb{R}^{d}} \widehat{f}(\xi)\,e^{i\xi\cdot x}\,d\xi,
    \qquad x\in\mathbb{R}^{d},
\end{equation*}
and this representation defines a continuous function. Writing $\widehat{f}(\xi)=A(\xi)e^{i\phi(\xi)}$ with $A(\xi)=|\widehat{f}(\xi)|\ge 0$ and $\phi(\xi)\in\mathbb{R}$, the preceding expression becomes
\begin{equation*}
    f(x)
    =
    \int_{\mathbb{R}^{d}}
    A(\xi)\,\cos\!\big(\xi\cdot x + \phi(\xi)\big)\,d\xi.
\end{equation*}
Define
\begin{equation*}
    \Lambda := \|\widehat{f}\|_{L^{1}(\mathbb{R}^{d})}
    = \int_{\mathbb{R}^{d}} A(\xi)\,d\xi,
\end{equation*}
and introduce the probability density 
\begin{equation*}
    p(\xi) := \frac{A(\xi)}{\Lambda}.
\end{equation*}
Then the Fourier representation of $f$ may be rewritten as an expectation:
\begin{equation*}
    f(x)
    =
    \Lambda\,\mathbb{E}_{\xi\sim p}\!\left[
        \cos\big(\xi\cdot x + \phi(\xi)\big)
    \right].
\end{equation*}
To obtain a Monte–Carlo approximation, let $\xi_1,\dots,\xi_n$ be independent samples drawn from $p$, and define
\begin{equation*}
    f_n(x):=\frac{\Lambda}{n}\sum_{j=1}^{n}\cos\!\big(\xi_j\cdot x + \phi(\xi_j)\big).
\end{equation*}
By construction, $\mathbb{E}[f_n(x)] = f(x) \;\text{for all } x\in\mathbb{R}^{d},$ so $f_n$ is an unbiased estimator of $f$ at each point. Since $|\cos(\cdot)|\le 1$, it follows that $    \operatorname{Var}\!\big(\cos(\xi\cdot x+\phi(\xi))\big)\le 1,$ and therefore
\begin{equation*}
    \operatorname{Var}\!\big(f_n(x)\big)=
    \frac{\Lambda^{2}}{n}\,
    \operatorname{Var}\!\big(\cos(\xi\cdot x+\phi(\xi))\big)
    \le
    \frac{\Lambda^{2}}{n}.
\end{equation*}
Integrating this pointwise variance bound with respect to $\mu$
\begin{equation*}
    \mathbb{E}\,\|f-f_n\|_{L^{2}(\mu)}^{2} =\int_{\mathbb{R}^{d}} \operatorname{Var}(f_n(x))\,d\mu(x)\le\frac{\Lambda^{2}}{n}.
\end{equation*}
Taking square roots yields $    \mathbb{E}\,\|f-f_n\|_{L^{2}(\mu)}\le\frac{\Lambda}{\sqrt{n}}.$ Since the expectation of a nonnegative random variable bounds its infimum over realizations, there exists at least one realization of the samples $\{\xi_1,\dots,\xi_n\}$ for which
\begin{equation*}
    \|f-f_n\|_{L^{2}(\mu)}
    \le
    \frac{\Lambda}{\sqrt{n}}
    =
    \frac{\|\widehat{f}\|_{L^{1}}}{\sqrt{n}}.
\end{equation*}
To realize $f_n$ with a QNN, observe that
\begin{equation*}
    f_n(x)
    =
    \sum_{j=1}^{n}
    w_j\,
    \cos\!\big(\xi_j\cdot x + \phi(\xi_j)\big),
    \qquad
    w_j = \frac{\Lambda}{n}.
\end{equation*}
By Lemma~\ref{lemma-1}, for each pair $(\xi_j,\phi(\xi_j))$ there exist QNN parameters for which a single qubit produces the expectation $\phi_j(x)=\cos\!\big(\xi_j\cdot x + \phi(\xi_j)\big).$ The same lemma guarantees that any coefficient $w_j>0$ can be incorporated into the classical readout weight assigned to that qubit. Thus, by assigning one qubit per cosine term and choosing classical readout weights $c_j=w_j$ and $c_0=0$, the resulting QNN output satisfies $    u_\theta(x)= f_n(x)\;\text{for all } x\in\mathbb{R}^{d}.$ Combining the identity $u_\theta=f_n$ with the earlier bound gives
\begin{equation*}
    \|f-u_\theta\|_{L^{2}(\mu)}
    =
    \|f-f_n\|_{L^{2}(\mu)}
    \le
    \frac{\|\widehat{f}\|_{L^{1}}}{\sqrt{n}},
\end{equation*}
which completes the proof.
\end{proof}

\begin{corollary}
Let $\mu$ be a probability measure on $\mathbb{R}^{d}$ and let 
$f \in C(\mathbb{R}^d)\cap L^1(\mathbb{R}^d)$.
Then for any $\varepsilon > 0$, there exist $n\in\mathbb{N}$ and $\theta=\{\boldsymbol{a},\boldsymbol{b},\Gamma,\boldsymbol{c},c_0\},$ such that the corresponding QNN output $u_{\theta}(\boldsymbol{x})$ satisfies
\begin{equation*}
    \|f(\boldsymbol{x}) - u_{\theta}(\boldsymbol{x})\|_{L^2(\mu)}
    \le \varepsilon.
\end{equation*}
\end{corollary}
\begin{proof}
The result follows immediately from Theorem~\ref{thm:l2_uat_concise} by choosing $n\in\mathbb{N}$ such that $\|\widehat f\|_{L^1}/\sqrt{n}\le \varepsilon$.
\end{proof}

\begin{remark}
The approximation bound in Theorem~\ref{thm:l2_uat_concise} implies that achieving an $\mathcal{L}^2$-error tolerance $\varepsilon>0$ requires 
$n=\mathcal{O}(\varepsilon^{-2})$ cosine ridge functions. Since each such ridge function is realized by a single qubit with a constant number of trainable parameters, the total number of trainable parameters in the QNN scales as $\mathcal{O}(\varepsilon^{-2})$. Although the approximation requires $\mathcal{O}(\varepsilon^{-2})$ trigonometric features, the number of physical qubits needed to represent these features grows only logarithmically with the target accuracy. By exploiting the finite Fourier structure established in Proposition~\ref{prop-1} and Corollary~\ref{coroll-1}, 
$\mathcal{O}(\lceil \log_2(\varepsilon^{-1}) \rceil)$ qubits suffice to generate the required trigonometric feature expansion, with the remaining complexity absorbed into the trainable circuit parameters.
\end{remark}

\begin{remark}
Denote $\mathcal{S}
:=\{ g\in C(\mathbb{R}^d)\cap L^1(\mathbb{R}^d) : \widehat{g}\in L^1(\mathbb{R}^d)\}.$
If $\mathcal{S}$ is dense in $L^2(\mathbb{R}^d,\mu)$, then for any
$f\in L^2(\mathbb{R}^d,\mu)$ and any $\varepsilon>0$ there exist
$n\in\mathbb{N}$ and QNN parameters $\theta$ such that $\| f - u_{\theta} \|_{L^2(\mu)} \le \varepsilon.$
In this case, the proposed QNN architecture is dense in $L^2(\mathbb{R}^d,\mu)$ and therefore constitutes a universal approximator for square-integrable functions with respect to the measure $\mu$.
\end{remark}
The results of this section establish that the proposed QNN architecture admits a finite trigonometric feature representation and satisfies explicit quantitative $\mathcal{L}^2$ approximation guarantees for a broad class of functions with integrable Fourier transforms. In particular, the approximation error decays at a rate $\mathcal{O}(n^{-1/2})$, while the required number of qubits grows only logarithmically with respect to the target accuracy. These properties provide a rigorous theoretical foundation for employing QNNs as expressive function approximators in scientific computing. In the following sections, we build upon this approximation theory to develop two main QPINN approaches for solving nonlinear IDEs and FIPDEs. Rather than focusing on expressivity, the subsequent methodology addresses the incorporation of nonlocal operators, integral terms, and fractional derivatives within a physics-informed quantum learning framework.

\section{Proposed Methodology}\label{methodology}

In this section, we present a QPINN-based framework for solving general nonlinear IDEs \eqref{model-general-IDE}, systems of IDEs \eqref{model-general-system}, and FPIDEs \eqref{model-general-FIPDE}, which serve as our model problems. We introduce two complementary strategies and provide a unified structure for addressing integro-differential and fractional integro-partial differential equations. The first strategy, termed Numerical-quadrature based QPINN (N-QPINN), approximates integral and fractional operators using numerical quadrature schemes such as the composite trapezoidal rule or Gaussian quadrature. These discretized expressions are incorporated directly into the loss function, mirroring the classical fPINN approach but leveraging quantum neural networks as the underlying function approximators. The second strategy, termed Auxiliary-function based QPINN (A-QPINN), reformulates integral operators by introducing auxiliary functions that represent the integral terms explicitly. These auxiliary functions are learned jointly with the primary solution within the QPINN architecture. While the classical A-PINN cannot be extended to FPIDEs due to the inherent nonlocality of fractional derivatives and their lack of closed-form differential reformulations. The following subsections provide detailed descriptions of these two approaches.

\subsection{Numerical-quadrature based QPINN (N-QPINN)}\label{subsec-4.1}

The N-QPINN provides a unified framework for approximating the integral and fractional operators appearing in the general models \eqref{model-general-IDE}, \eqref{model-general-system}, and \eqref{model-general-FIPDE}. As illustrated in Figure~\ref{fig:qnn_NQPINN}, the variational quantum circuit maps the inputs $\boldsymbol{x}$ or $(\boldsymbol{x},t)$ to the QNN output, which serves as the approximate solution $u_\theta(\boldsymbol{x})$ or $u_\theta(\boldsymbol{x}, t)$. In this formulation, all nonlocal operators, including spatial, spatio-temporal, and kernel-based integrals, as well as fractional integrals and derivatives, are approximated using suitable numerical quadrature schemes (e.g., trapezoidal, Simpson’s, Gauss--Legendre, or tensor-product rules). In contrast, integer-order differential operators are evaluated using automatic differentiation applied directly to the QNN output. This hybrid strategy enables a flexible and unified treatment of nonlocal and fractional terms, making the N-QPINN framework applicable to nonlinear IDEs, coupled IDE systems, and fractional integro–PDEs. The governing equations and associated constraints are enforced through the composite loss function
\begin{equation}\label{total-loss}
    \mathcal{L}_{\mathrm{Total}}
    =
    \mathcal{L}_{\mathrm{PDE}}
    +
    \mathcal{L}_{\mathrm{BC}}
    +
    \mathcal{L}_{\mathrm{IC}}.
\end{equation}
where
\begin{equation}
\begin{cases}
\displaystyle
\mathcal{L}_{\mathrm{PDE}}
=
\frac{1}{N_R}
\sum_{i=1}^{N_R}
\bigl|\widehat{\mathcal{R}}(\boldsymbol{x}_i)\bigr|^2, \\[10pt]
\displaystyle
\mathcal{L}_{\mathrm{BC}}
=
\frac{1}{N_B}
\sum_{i=1}^{N_B}
\bigl|\mathscr{B}[u_\theta](\boldsymbol{x}_i) - g(\boldsymbol{x}_i)\bigr|^2, \\[10pt]
\displaystyle
\mathcal{L}_{\mathrm{IC}}
=
\frac{1}{N_I}
\sum_{i=1}^{N_I}
\bigl|\partial^{\boldsymbol{k}} u_\theta(\boldsymbol{x}_i) - a_{\boldsymbol{k}}\bigr|^2.
\end{cases}
\end{equation}
For the nonlinear IDE in Model Problem~\eqref{model-general-IDE}, the integral operator can be approximated as
\begin{equation}\label{eq:quadrature-general}
\int_{\mathcal{D}(\boldsymbol{x})}
\mathscr{K}(\boldsymbol{x},\boldsymbol{z})\,
\mathcal{N}(u_\theta(\boldsymbol{z}))\, d\boldsymbol{z}
\;\approx\;
\sum_{p=1}^{N_q}
w_p\,
\mathscr{K}(\boldsymbol{x},\boldsymbol{z}_p)\,
\mathcal{N}(u_\theta(\boldsymbol{z}_p)),
\end{equation}
and the corresponding PDE residual at $\boldsymbol{x}_i$ becomes
\begin{equation}
\widehat{\mathcal{R}}(\boldsymbol{x}_i)
=
\mathscr{L}[u_\theta(\boldsymbol{x}_i)]
-
f(\boldsymbol{x}_i)
-
\lambda
\sum_{p=1}^{N_q}
w_p\, 
\mathscr{K}(\boldsymbol{x}_i,\boldsymbol{z}_p)\,
\mathcal{N}(u_\theta(\boldsymbol{z}_p)).
\end{equation}
Similarly, for the coupled system of IDEs in Model Problem~\eqref{model-general-system}, each component equation yields the residual
\begin{equation}
\begin{aligned}
\widehat{\mathcal{R}}_i(\boldsymbol{x}_i)
&=
\mathscr{L}_i[u_{1,\theta},\dots,u_{m,\theta}](\boldsymbol{x}_i)
-
f_i(\boldsymbol{x}_i)\\
&\qquad-\lambda_i
\sum_{p=1}^{N_q}
w_p\,
\mathscr{K}_i(\boldsymbol{x}_i,\boldsymbol{z}_p)\,
\mathcal{N}_i\!\left(
u_{1,\theta}(\boldsymbol{z}_p),
\dots,
u_{m,\theta}(\boldsymbol{z}_p)
\right),
\;
i=1,\dots,m.
\end{aligned}
\end{equation}
For the fractional integro–PDE \eqref{model-general-FIPDE}, the fractional operators are discretized as
\[
D_t^\alpha u_\theta(\boldsymbol{x},t)
\approx 
\sum_{q=1}^{N_t} \omega_q^{(\alpha)}
\bigl[u_\theta(\boldsymbol{x},t_q)-u_\theta(\boldsymbol{x},0)\bigr],
\qquad
\mathcal{I}_t^\beta u_\theta(\boldsymbol{x},t)
\approx 
\sum_{q=1}^{N_t} v_q^{(\beta)}\,u_\theta(\boldsymbol{x},t_q),
\]
and the space–time nonlocal term is approximated by
\begin{equation}
\begin{aligned}
\iint\limits_{\mathcal{D}(\boldsymbol{x},t)}
\mathscr{K}((\boldsymbol{x},t),(\boldsymbol{z},s))
\mathcal{N}(u_\theta(\boldsymbol{z},s))\, d\boldsymbol{z}\,ds
\approx\;
\sum_{p,q}
w_p w_q\,
\mathscr{K}((\boldsymbol{x},t),(\boldsymbol{z}_p,s_q))
\mathcal{N}(u_\theta(\boldsymbol{z}_p,s_q)).
\end{aligned}
\end{equation}
Across all three model problems, these quadrature-based discretizations convert integral and fractional operators into finite weighted sums that interface seamlessly with the QPINN loss. After assembling $\mathcal{L}_{\mathrm{Total}}$, a classical optimizer updates the quantum parameters $\theta$ in the hybrid loop until convergence, completing the N-QPINN training process illustrated in Figure~\ref{fig:qnn_NQPINN}.
\begin{figure}[ht] 
\centering 
\includegraphics[width=1\linewidth]{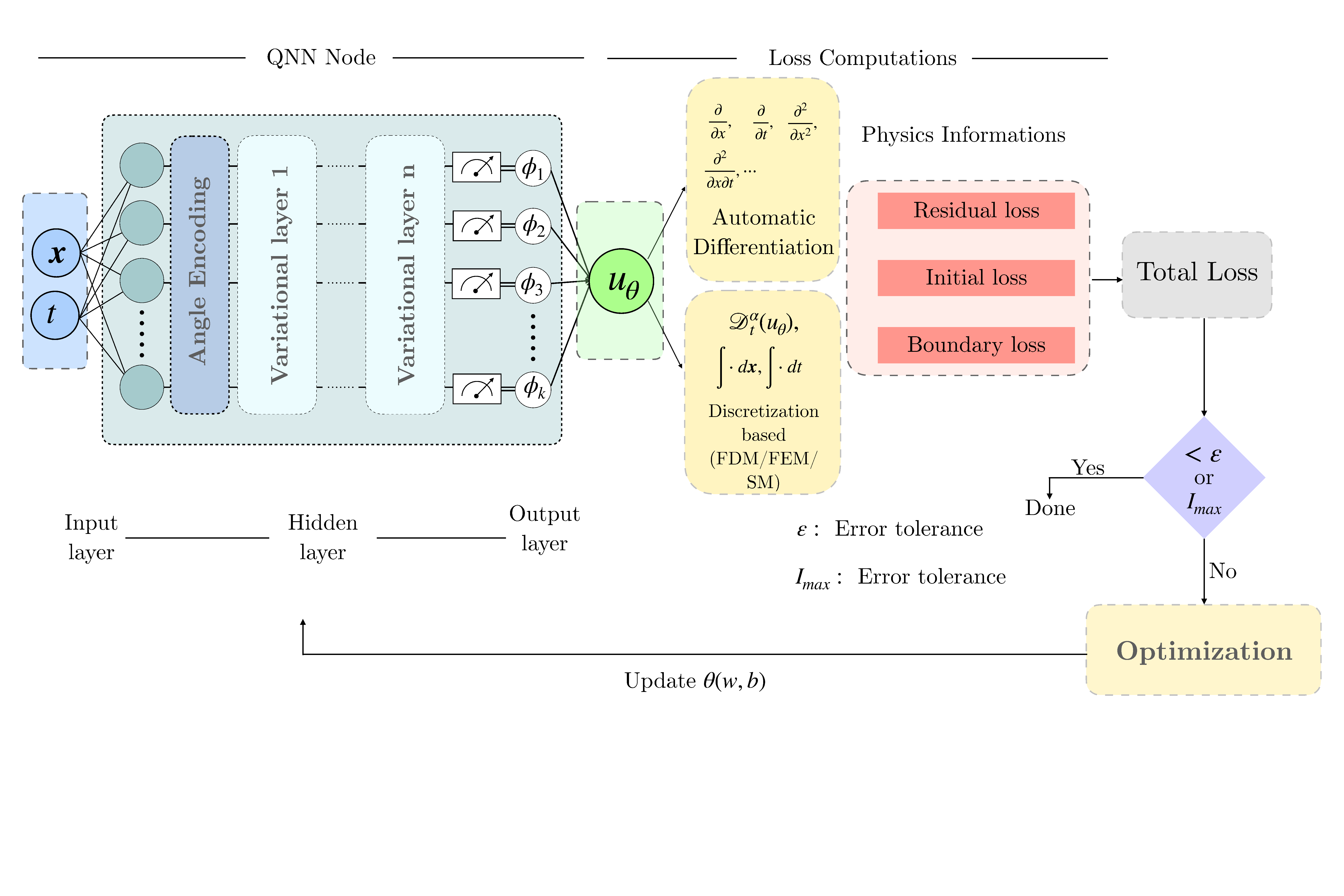} 
\caption{General N-QPINN architecture for solving model problems~\eqref{model-general-IDE},\eqref{model-general-system} and \eqref{model-general-FIPDE} involving nonlocal and fractional operators.} 
\label{fig:qnn_NQPINN}
\end{figure}
\subsection{Auxiliary-function based QPINN (A-QPINN)}\label{subsec-4.2}
The A-QPINN provides a quadrature-free mechanism for solving the nonlocal integro–differential models~\eqref{model-general-IDE} and~\eqref{model-general-system} by introducing auxiliary fields that represent the action of the integral operators, thereby converting the governing equations into a fully local system of coupled differential equations. As shown in Figure~\ref{fig:AQPINN_architecture}, the variational quantum circuit generates a multi-output approximation $(u_\theta(\boldsymbol{x}), v_\theta(\boldsymbol{x}), w_\theta(\boldsymbol{x}))$ that simultaneously predicts the primary solution and the auxiliary variables. Automatic differentiation is then used to evaluate every differential operator in the transformed local formulation, eliminating the need for numerical quadrature and significantly accelerating training. Because fractional operators in Model Problem~\eqref{model-general-FIPDE} cannot be rewritten as local consistency relations, the A-QPINN methodology applies only to Models~\eqref{model-general-IDE}–\eqref{model-general-system}.
For Model Problem~\eqref{model-general-IDE}, the nonlocal integral term is replaced by the auxiliary function
\[
v(\boldsymbol{x})
:=
\int_{\mathcal{D}(\boldsymbol{x})}
\mathscr{K}(\boldsymbol{x},\boldsymbol{z})\,
\mathcal{N}(u(\boldsymbol{z}))\,d\boldsymbol{z},
\]
which reformulates \eqref{model-general-IDE} into the coupled local system
\begin{equation}\label{AQPINN-MI-system}
\begin{cases}
\displaystyle 
\mathscr{L}[u(\boldsymbol{x})]
=
f(\boldsymbol{x}) + \lambda\, v(\boldsymbol{x}),
\qquad \boldsymbol{x}\in\Omega, \\[10pt]
\displaystyle
\mathscr{N}[v(\boldsymbol{x})]
=
h(\boldsymbol{x},u(\boldsymbol{x})), \\[10pt]
\displaystyle 
\partial^{\boldsymbol{k}} u(\boldsymbol{x}_0)=a_{\boldsymbol{k}},
\qquad |\boldsymbol{k}|<N, \\[6pt]
\displaystyle 
\partial^{\boldsymbol{\gamma}} v(\boldsymbol{x}_0)=a_{\boldsymbol{\gamma}},
\qquad |\boldsymbol{\gamma}|<P, \\[6pt]
\displaystyle 
\mathscr{B}_j[u(\boldsymbol{x})]=g_j(\boldsymbol{x}),
\qquad 
\mathscr{C}_j[v(\boldsymbol{x})]=c_j(\boldsymbol{x}),
\quad \boldsymbol{x}\in\partial\Omega,\;\; j=1,\dots,n.
\end{cases}
\end{equation}
For the nonlinear coupled system in Model Problem~\eqref{model-general-system}, the same construction is performed componentwise. Introducing
\[
v_i(\boldsymbol{x})
=
\int_{\mathcal{D}(\boldsymbol{x})}
\mathscr{K}_i(\boldsymbol{x},\boldsymbol{z})\,
\mathcal{N}_i(u_1(\boldsymbol{z}),\dots,u_m(\boldsymbol{z}))\, d\boldsymbol{z},
\qquad i=1,\dots,m,
\]
the integro–differential system becomes
\begin{equation}\label{AQPINN-MII-system}
\begin{cases}
\displaystyle
\mathscr{L}_i[u_1,\dots,u_m](\boldsymbol{x})
=
f_i(\boldsymbol{x}) + \lambda_i\, v_i(\boldsymbol{x}),
\qquad \boldsymbol{x}\in\Omega,\;\; i=1,\dots,m, \\[10pt]
\displaystyle
\mathscr{N}_i[v_i(\boldsymbol{x})]
=
h_i(\boldsymbol{x},u_1(\boldsymbol{x}),\dots,u_m(\boldsymbol{x})),
\qquad i=1,\dots,m, \\[10pt]
\displaystyle
\partial^{\boldsymbol{k}} u_i(\boldsymbol{x}_0)=a_{i,\boldsymbol{k}},
\qquad |\boldsymbol{k}|<N_i, \\[6pt]
\displaystyle
\partial^{\boldsymbol{\gamma}} v_i(\boldsymbol{x}_0)=a_{i,\boldsymbol{\gamma}},
\qquad |\boldsymbol{\gamma}|<P_i, \\[6pt]
\displaystyle
\mathscr{B}_{i,j}[u_i(\boldsymbol{x})]=g_{i,j}(\boldsymbol{x}),
\qquad
\mathscr{C}_{i,j}[v_i(\boldsymbol{x})]=c_{i,j}(\boldsymbol{x}),
\quad \boldsymbol{x}\in\partial\Omega,\;\; j=1,\dots,n_i.
\end{cases}
\end{equation}
The hybrid quantum–classical model enforces the coupled PDE systems \eqref{AQPINN-MI-system}–\eqref{AQPINN-MII-system} through the total loss
\[
\mathcal{L}_{\mathrm{Total}}
=
\mathcal{L}_{\mathrm{PDE}}
+
\mathcal{L}_{\mathrm{Aux}}
+
\mathcal{L}_{\mathrm{BC}}
+
\mathcal{L}_{\mathrm{IC}},
\]
where the physics residuals for Model Problem~\eqref{AQPINN-MI-system} are
\[
\mathcal{L}_{\mathrm{PDE}}
=
\frac{1}{N_R}\sum_{i=1}^{N_R}
\Big|
\mathscr{L}[u_\theta(\boldsymbol{x}_i)]
-
f(\boldsymbol{x}_i)
-
\lambda\, v_\theta(\boldsymbol{x}_i)
\Big|^2,
\
\mathcal{L}_{\mathrm{Aux}}
=
\frac{1}{N_R}\sum_{i=1}^{N_R}
\Big|
\mathscr{N}[v_\theta(\boldsymbol{x}_i)]
-
h(\boldsymbol{x}_i,u_\theta(\boldsymbol{x}_i))
\Big|^2.
\]
The boundary and initial losses take the form
\[
\mathcal{L}_{\mathrm{BC}}
=
\frac{1}{N_B}\sum_{i=1}^{N_B}
\Big(
|\mathscr{B}[u_\theta](\boldsymbol{x}_i)-g(\boldsymbol{x}_i)|^2
+
|\mathscr{C}[v_\theta](\boldsymbol{x}_i)-c(\boldsymbol{x}_i)|^2
\Big),
\]
\[
\mathcal{L}_{\mathrm{IC}}
=
\frac{1}{N_I}\sum_{i=1}^{N_I}
\Big(
|\partial^{\boldsymbol{k}}u_\theta(\boldsymbol{x}_i)-a_{\boldsymbol{k}}|^2
+
|\partial^{\boldsymbol{\gamma}}v_\theta(\boldsymbol{x}_i)-a_{\boldsymbol{\gamma}}|^2
\Big).
\]
For the multi-field system \eqref{AQPINN-MII-system}, the loss components extend naturally to
\[
\mathcal{L}_{\mathrm{PDE}}
=
\sum_{i=1}^{m}
\frac{1}{N_R}
\sum_{j=1}^{N_R}
\Big|
\mathscr{L}_i[u_{1,\theta},\dots,u_{m,\theta}](\boldsymbol{x}_j)
-
f_i(\boldsymbol{x}_j)
-
\lambda_i v_{i,\theta}(\boldsymbol{x}_j)
\Big|^2,
\]
\[
\mathcal{L}_{\mathrm{Aux}}
=
\sum_{i=1}^{m}
\frac{1}{N_R}
\sum_{j=1}^{N_R}
\Big|
\mathscr{N}_i[v_{i,\theta}(\boldsymbol{x}_j)]
-
h_i(\boldsymbol{x}_j,u_{1,\theta}(\boldsymbol{x}_j),\dots,u_{m,\theta}(\boldsymbol{x}_j))
\Big|^2,
\]
and analogous boundary and initial losses for all components. A classical optimizer updates the variational parameters $\theta$ until convergence. Because the A-QPINN avoids numerical quadrature and relies exclusively on automatic differentiation, it offers a highly efficient and scalable framework for solving nonlocal integro–differential models \eqref{model-general-system},\eqref{model-general-IDE}.
\begin{figure}[ht]
    \centering
    \includegraphics[width=1\linewidth]{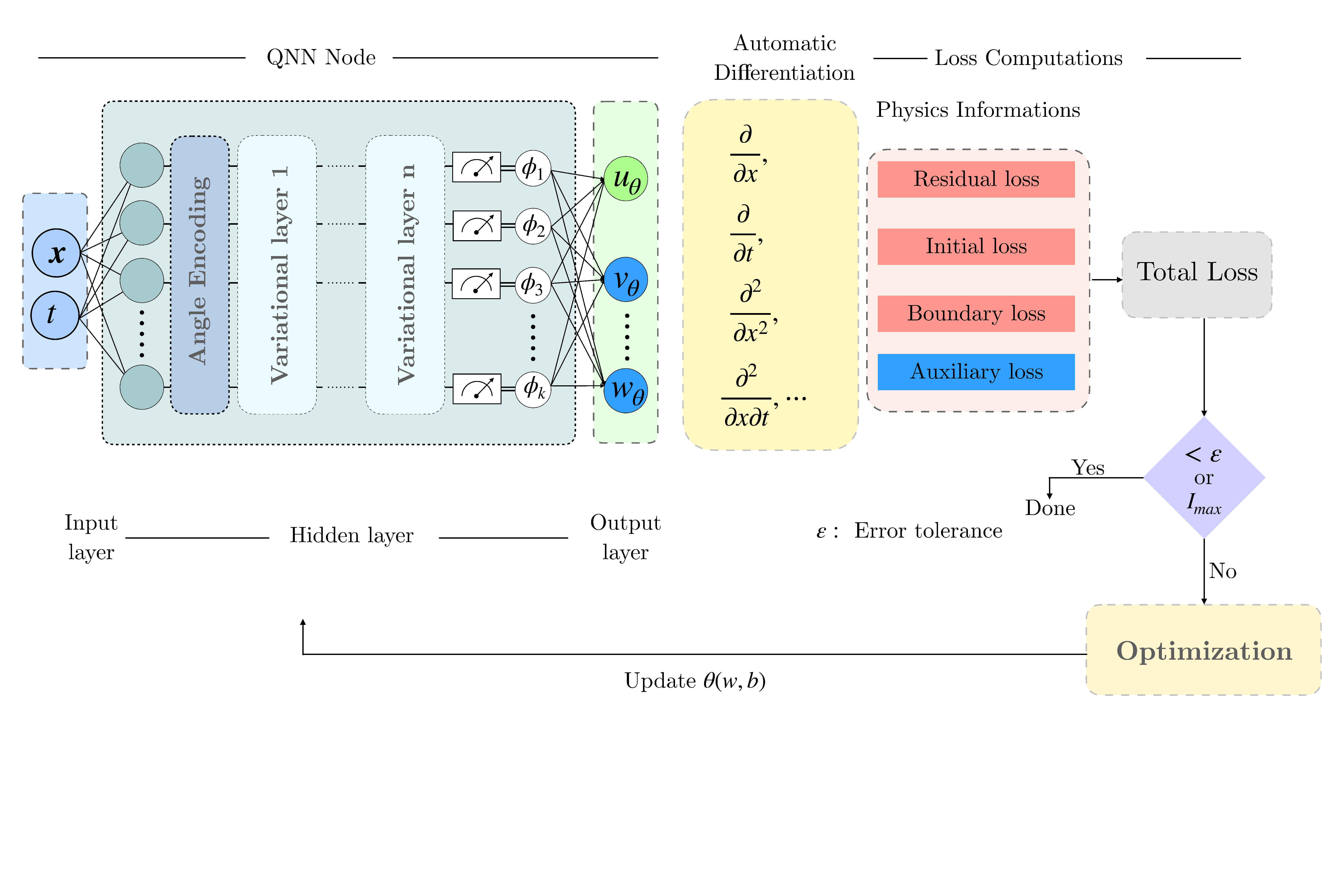}
    \caption{General A-QPINN architecture for solving Models~\eqref{model-general-IDE} and \eqref{model-general-system} via auxiliary-function reformulation of integral operators operators.}
    \label{fig:AQPINN_architecture}
\end{figure}
\section{Numerical Experiments}\label{num-exp} 
All neural network weights were initialized using the Xavier initialization to promote a stable gradient propagation during training. For quantitative evaluation, the relative $\mathcal{L}_2$-error is computed over a uniformly sampled validation set, defined as
\begin{equation}
    \text{Relative }\mathcal{L}_2\text{-error} = \frac{\| u_{\theta}(\boldsymbol {x}) - u_{\text{ref}}(\boldsymbol {x}) \|_{L^2(\Omega)}}{\| u_{\text{ref}}(\boldsymbol {x}) \|_{L^2(\Omega)}},
\end{equation}
where $u_{\theta}(\boldsymbol {x})$ denotes the network prediction and $u_{\text{ref}}(\boldsymbol {x})$ represents the reference solution. All simulations were conducted using the open-source frameworks \texttt{PennyLane} version $0.40.0$ and \texttt{PyTorch} version $2.6.0$.
 Quantum circuits were constructed and simulated using PennyLane’s built-in \texttt{default.qubit} device, which is a noiseless state-vector simulator designed for ideal quantum circuit simulations on qubit-based architectures. Implementation of the N-QPINN and A-QPINN leveraged PennyLane’s quantum node interface. Numerical experiments were performed on an NVIDIA RTX A4500 GPU using CUDA version 12.4. Given the sensitivity of the neural network training to weight initialization, the reported results correspond to the average over 5 independent training runs to ensure statistical robustness.
\subsection{One-dimensional nonlinear Volterra equation}\label{Example-1}
    In our first experiment, we investigate the capability of the proposed method in solving a one-dimensional nonlinear Volterra integro-differential equation. The benchmark problem is given by
\begin{equation}\label{ex1-1dVolterra}
    \frac{du(x)}{dx} + u(x) 
    = \lambda \int_{0}^{x} e^{\,t - x} \, y(t) \, dt, 
    \quad x \in [0,5],
\end{equation}
with the initial condition \( u(0) = 1 \) and parameter \(\lambda = 1\). The exact solution of this problem is $u(x) = e^{-x} \cosh(x)$.
We solve this IDE using both the methods N-QPINN and A-QPINN described in subsections \ref{subsec-4.1} and \ref{subsec-4.2}, respectively. 
For solving IDE \eqref{ex1-1dVolterra} with the N-QPINN approach, the first step is to approximate the integral term using a numerical quadrature scheme. 
In particular, we employ the Gaussian--Legendre quadrature of degree $20$ to approximate
\begin{equation}\label{ex1-quadrature}
    \int_{0}^{x} e^{\,t-x} u(t) \, dt 
    \;\approx\; \sum_{i=1}^{20} w_i \, e^{\,t_i(x)-x} \, u\!\bigl(t_i(x)\bigr), \qquad t_i(x) = \tfrac{x}{2}\,(\xi_i+1),
\end{equation}
where \( \{ w_i \}_{i=1}^{20} \) and \( \{ \xi_i \}_{i=1}^{20} \) denote the Gauss--Legendre weights and nodes on $[-1,1]$, respectively, and $t_i(x)$
are the quadrature points mapped to the interval \([0,x]\). Substituting \eqref{ex1-quadrature} into \eqref{ex1-1dVolterra}, the original integro-differential equation reduces to
\begin{equation}
    \displaystyle\dfrac{du(x)}{dx} + u(x) 
    \;\approx\; \lambda \sum_{i=1}^{20} w_i \, e^{\,t_i(x)-x} \, u\!\bigl(t_i(x)\bigr),
    \quad x \in [0,5], \quad u(0) = 1.
\end{equation}
This reduced form is then used in the loss construction for the N-QPINN method. Let say QNN approximate $u(x)$ as $u_{\theta}(x)$ then the residual at interior collocation points $x_j \in [0,5]$ is 
\begin{equation}
\widehat{\mathcal{R}}(x_j) := \displaystyle u'_{\theta}(x_j) + u_{\theta}(x_j) 
    \;-\; \lambda \sum_{i=1}^{20} w_i \, e^{\,t_i(x_j)-x_j} \, u_{\theta}\!\bigl(t_i(x_j)\bigr),
\end{equation}
So the total loss is constructed as 
$\mathcal{L}_{total} = \mathcal{L}_{IDE}+\mathcal{L}_{IC}$.
where,
\begin{equation}
\displaystyle\mathcal{L}_{IDE} = \frac{1}{N_R}\sum_{j=1}^{N_R} \bigl|\widehat{\mathcal{R}}(x_j)\bigr|^2,\qquad
          \mathcal{L}_{IC} = |\hat{u}(0;\theta)-u(0)|^2.
\end{equation}
To solve IDE \eqref{ex1-1dVolterra} with A-QPINN aprroach, first we define an auxiliary function
$v(x) = \int_{0}^{x} e^{\,t-x} \, u(t) \, dt,$
representing the integral term in \eqref{ex1-1dVolterra}. Differentiating \( v(x) \) with respect to \( x \) by applying Leibniz’s rule yields
\begin{equation*}
    \frac{dv(x)}{dx} 
    = u(x) - \int_{0}^{x} e^{\,t-x} \, u(t) \, dt
    = u(x) - v(x).
\end{equation*}
Hence, the original Volterra IDE~\eqref{ex1-1dVolterra} is equivalent to the first-order system of differential equation 
\begin{equation}\label{ex1-system}
    \left\{
    \begin{array}{ll}
          \dfrac{du(x)}{dx} = \lambda v(x) - u(x), \quad x \in [0, 5], \\[10pt]
          \dfrac{dv(x)}{dx} = u(x) - v(x),\quad u(0) = 1, \; v(0) = 0. \\[10pt]
          
     \end{array}
    \right.
\end{equation}
Then a two-output QNN is implemented to approximate  $u(x)$ and $v(x)$ in \eqref{ex1-system}. The total loss corresponding to this approach is constructed as $\mathcal{L}_{total} = \mathcal{L}_{IDE}+\mathcal{L}_{Aux} +\mathcal{L}_{IC},$ where 
\begin{equation}
    \left\{
    \begin{array}{ll}
              \displaystyle\mathcal{L}_{IDE} = \frac{1}{N_R}\sum_{j=1}^{N_R} \bigl| u'_{\theta}(x_j)-\lambda v_{\theta}(x_j)+ u_{\theta}(x_j)\bigr|^2,  \\[10pt]
              \displaystyle\mathcal{L}_{Aux} = \frac{1}{N_R}\sum_{j=1}^{N_R} \bigl| v_{\theta}'(x_j)- u_{\theta}(x_j) - v_{\theta}(x_j)\bigr|^2,\\[20pt]
          \mathcal{L}_{IC} = |u_{\theta}(0)-u(0)|^2+|v_{\theta}(0)-v(0)|^2.\\[10pt]
     \end{array}
    \right.
\end{equation}
 In this experiment, both quantum approaches A-QPINN and N-QPINN employ a QNN architecture consisting of $4$ variational layers and $8$ qubits. For comparison, we adopt a classical A-PINN architecture similar in size to the quantum models, following~\cite{Yuan2022APINN}, configured with $4$ hidden layers containing $8$ neurons each. All models are trained using $50$ collocation points randomly sampled from the problem domain, and the predictive accuracy is assessed using the relative $\mathcal{L}_2$-error.
\begin{table}[H]
\centering
\caption{Comparison of different methods for solving \eqref{ex1-1dVolterra}.}
\begin{tabular}{lccc}
\toprule
Method & Size& $\mathcal{L}_2$-error &  Trainable parameters \\
\midrule
Classical A-PINN~\cite{Yuan2022APINN}  & $8\times 4$ &  $1.09 \pm 0.32\times 10^{-3}$ & $322$ \\
N-QPINN (This work)                    & $8\times 4$  & $2.591 \pm 1.062 \times 10^{-4}$  & $121$ \\
A-QPINN (This work)                    & $8\times 4$  & $5.019 \pm 1.765 \times 10^{-4}$  & $130$ \\
\bottomrule
\label{tab:error_comparison}
\end{tabular}
\end{table}
\begin{figure}[ht]
    \centering
    \includegraphics[width=1\linewidth]{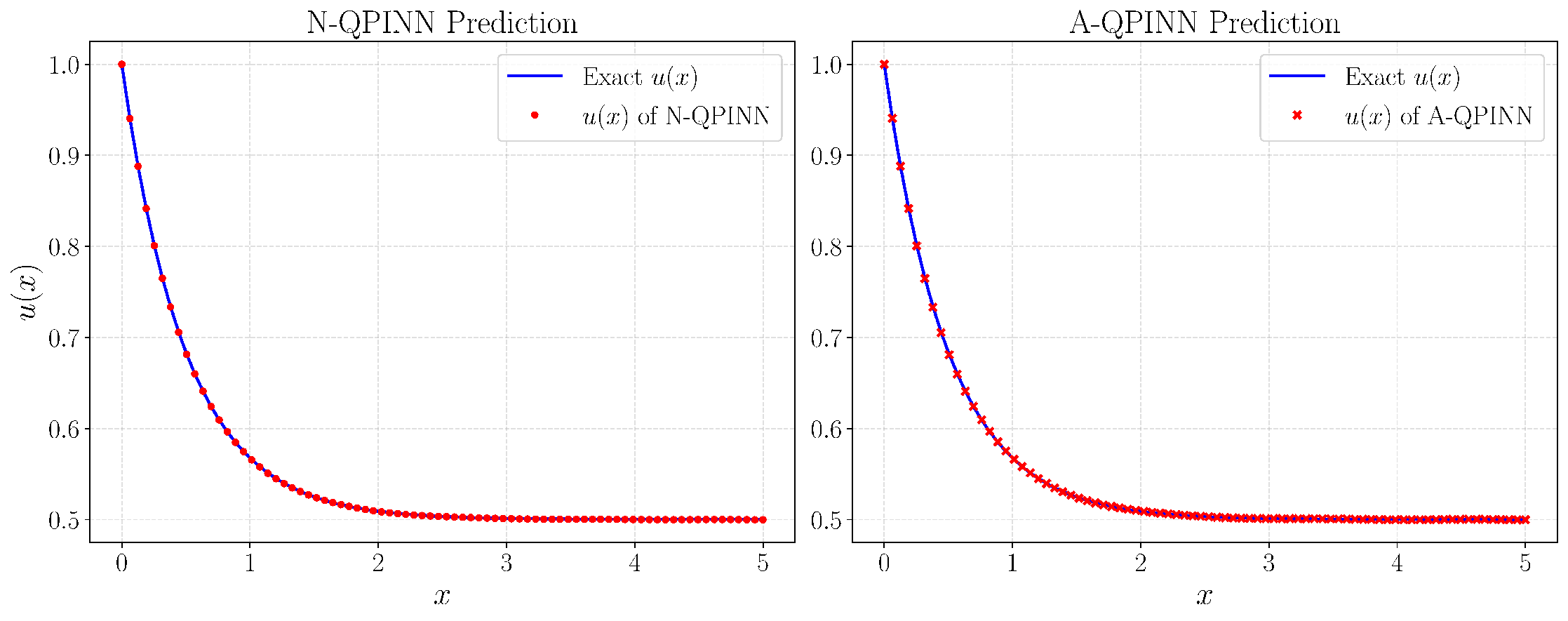}
    \caption{Predicted solutions of the 1D Volterra IDE \eqref{Example-1} using N-QPINN (left) and A-QPINN (right). The solid blue line represents the exact solution, while the red markers denote the network predictions.}
    \label{fig:Ex1-ExactvsPred}
\end{figure}
The comparison between the predicted solutions of the Volterra IDE~\eqref{ex1-1dVolterra} obtained using the proposed A-QPINN and N-QPINN models and the corresponding exact solution is presented in Figure~\ref{fig:Ex1-ExactvsPred}. A quantitative comparison of the methods is presented in Table~\ref{tab:error_comparison}. Both N-QPINN and A-QPINN outperform the classical A-PINN, demonstrating improved accuracy of the proposed quantum architectures. Among them, N-QPINN achieves the best performance, while A-QPINN also provides a clear improvement over the classical baseline.
\begin{figure}[ht]
    \centering
    \includegraphics[width=1\linewidth]{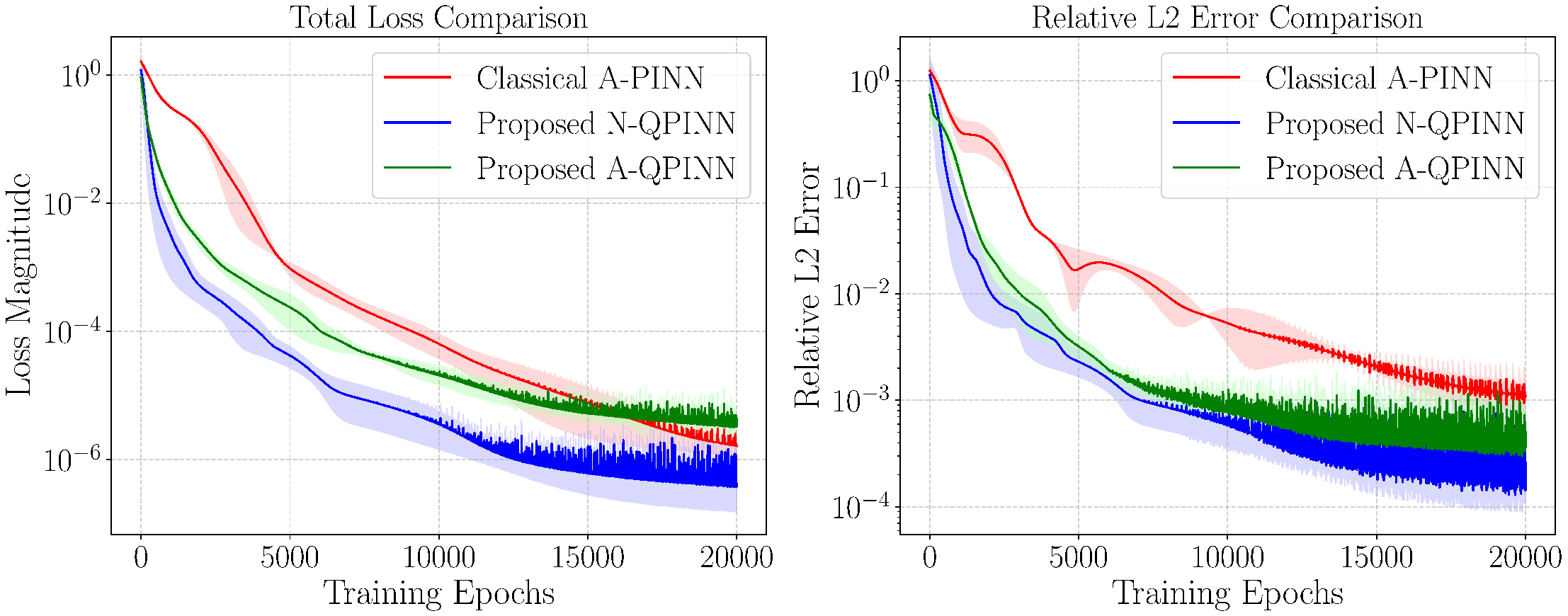}
    \caption{Total Loss (left) and relative $\mathcal{L}_2$-error (right) comparison of the 1D Volterra IDE \eqref{Example-1} with classical A-PINN,  N-QPINN  and A-QPINN.}
    \label{fig:Ex1-lossl2err-comp}
\end{figure}
Figure~\ref{fig:Ex1-lossl2err-comp} compares the training behavior of the classical A-PINN with the proposed N-QPINN and A-QPINN for solving the 1D Volterra IDE~\eqref{ex1-1dVolterra}. The left panel shows that both quantum models achieve substantially faster loss reduction and attain lower overall loss values than the classical A-PINN, with N-QPINN exhibiting the most efficient convergence. The right panel displays the evolution of the relative $\mathcal{L}_2$-error, where again the quantum approaches consistently outperform the classical baseline throughout training. In particular, the N-QPINN achieves the lowest and most stable error levels, followed by the A-QPINN, whereas the classical A-PINN converges more slowly and maintains comparatively higher error.

\begin{figure}[ht]
    \centering
    \includegraphics[width=0.9\linewidth]{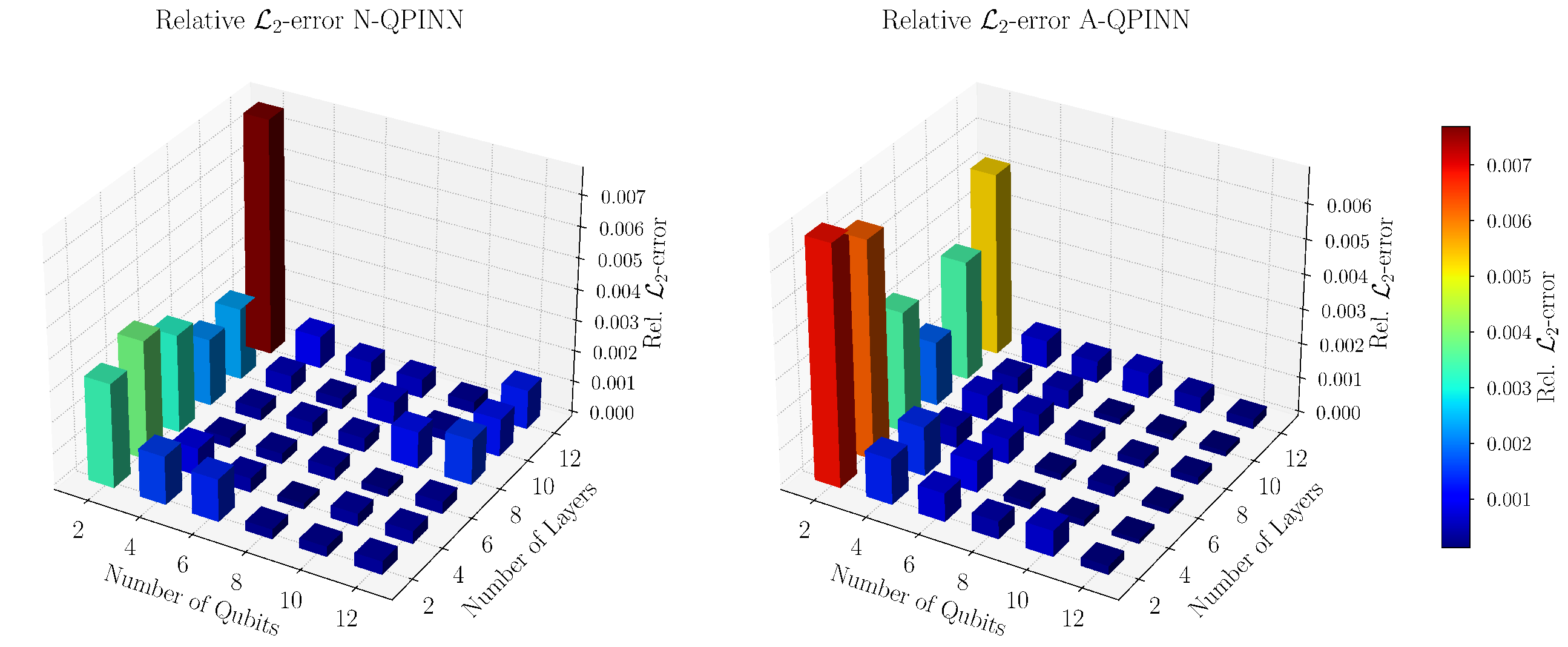}
    \caption{Relative $\mathcal{L}_2$-error of the 1D Volterra IDE~\eqref{ex2-volterra-original} for different QNN architectures obtained by varying the number of variational layers and qubits. Left: N-QPINN; right: A-QPINN.}
    \label{fig:Ex1-3dbarplot}
\end{figure}

To further evaluate the performance of the proposed quantum models, we analyze their predictive accuracies across a range of QNN architectures by varying both the number of variational layers and the number of qubits. Figure~\ref{fig:Ex1-3dbarplot} presents the resulting relative $\mathcal{L}_2$-errors in the form of 3D bar plots. In these experiments, the number of variational layers is varied from $2$ to $12$, and the number of qubits is likewise varied from $2$ to $12$. For the A-QPINN, the associated classical neural network component is fixed to four hidden layers with one hundred neurons in each layer. Across all configurations, the relative $\mathcal{L}_2$-error remains below $10^{-3}$, demonstrating the robustness of both quantum models. The results indicate that increasing the number of qubits per variational layer generally improves solution accuracy, whereas increasing the number of variational layers results in a slight degradation in accuracy. These empirical findings align with theoretical expectations derived from the universal approximation capabilities of QNN.

\subsection{System of nonlinear Volterra integro-differential equation}\label{eexample-2}
Consider the following nonlinear Volterra integro-differential system:
\begin{equation}\label{ex2-volterra-original}
\left\{
\begin{aligned}
    \frac{d^{2}u(x)}{dx^{2}} &= 1 - \frac{x^{3}}{3} - \frac{1}{2} \frac{dv(x)}{dx} 
    + \mu_{1} \int_{0}^{x} \big[ u^{2}(t) + v^{2}(t) \big]\, dt, \\[6pt]  
    \frac{d^{2}v(x)}{dx^{2}} &= -1 + x^{2} - x\, u(x) 
    + \mu_{2} \int_{0}^{x} \big[ u^{2}(t) - v^{2}(t) \big]\, dt, \\[6pt]
    u(0) &= 1, \quad v(0) = -1, \quad \frac{du}{dx}\bigg|_{x=0} = 2, \quad \frac{dv}{dx}\bigg|_{x=0} = 0, 
    \quad x \in [0,1],
\end{aligned}
\right.
\end{equation}
where $\mu_{1}$ and $\mu_{2}$ are prescribed constants. For the choice $\mu_{1} = \tfrac{1}{2}$ and $\mu_{2} = \tfrac{1}{4}$, the exact solution is given by $u(x) = x + e^{x}, \; v(x) = x - e^{x}.$ We solve this system using the N-QPINN and A-QPINN approaches introduced earlier in subsections \ref{subsec-4.1} and \ref{subsec-4.2}, respectively.  For the N-QPINN formulation, the IDEs are handled by numerically approximating the integral operators. The integral quantities
\begin{equation*}
    I_{1}(x) = \int_{0}^{x} \bigl[u^{2}(t)+v^{2}(t)\bigr]\, dt,
\qquad
I_{2}(x) = \int_{0}^{x} \bigl[u^{2}(t)-v^{2}(t)\bigr]\, dt,
\end{equation*}
are evaluated at each collocation point \(x_j\) using a composite trapezoidal quadrature rule.  
For a given \(x_j\), the interval \([0,x_j]\) is partitioned into \(M\) uniform subintervals with step size \(h_j = x_j/M\) and nodes \(t_k = k h_j\), \(k=0,\ldots,M\).  
The corresponding quadrature approximation is
\[
I(x_j) \approx h_j \left[ \tfrac{1}{2} f(t_0)
+ \sum_{k=1}^{M-1} f(t_k)
+ \tfrac{1}{2} f(t_M) \right], \qquad
f(t)=u_\theta^{2}(t)\pm v_\theta^{2}(t),
\]
Substituting these approximations into the governing equations yields the discrete residuals
\[
\widehat{\mathcal{R}}_1(x_j)
=
u_{\theta}''(x_j)
 - \left(
     1 - \frac{x_j^3}{3} - \frac12 v_{\theta}'(x_j)
     + \mu_1\, \widehat{I}_1(x_j)
   \right),
\]
\[
\widehat{\mathcal{R}}_2(x_j)
=
v_{\theta}''(x_j)
 - \left(
     -1 + x_j^2 - x_j u_{\theta}(x_j)
     + \mu_2\, \widehat{I}_2(x_j)
   \right),
\]
where \(\widehat{I}_{1,2}(x_j)\) denote the quadrature-based estimates of the Volterra integrals.  
The N-QPINN loss function then consists of the mean squared residual error over all collocation points together with the initial conditions. To solve the same system using the A-QPINN formulation, we avoid numerical quadrature by introducing auxiliary variables
\begin{equation*}
    w_1(x)=\int_{0}^{x} \bigl[u^{2}(t)+v^{2}(t)\bigr]\,dt,
\qquad
w_2(x)=\int_{0}^{x} \bigl[u^{2}(t)-v^{2}(t)\bigr]\,dt.
\end{equation*}
Differentiation with respect to \(x\) leads to the first-order relations
\begin{equation*}
\begin{aligned}
    w_1'(x)&=u^{2}(x)+v^{2}(x), \qquad w_1(0)=0,\\
    w_2'(x)&=u^{2}(x)-v^{2}(x), \qquad w_2(0)=0.
    \end{aligned}
\end{equation*}
Substituting these auxiliary variables into the original equations produces the equivalent system
\begin{equation*}
    \left\{
\begin{aligned}
u''(x) &= 1 - \frac{x^{3}}{3} - \frac12 v'(x) + \mu_1 w_1(x), \\[2pt]
v''(x) &= -1 + x^{2} - x u(x) + \mu_2 w_2(x), \\[2pt]
w_1'(x) &= u^{2}(x)+v^{2}(x), \\[2pt]
w_2'(x) &= u^{2}(x)-v^{2}(x),
\end{aligned}
\right.
\end{equation*}
which is enforced by a four-output QNN. In this case, the loss function contains the governing differential residuals, the auxiliary relations, and all initial conditions.

To assess the performance of the proposed QPINN approaches, all methods are trained using \(50\) collocation points uniformly sampled from the domain \([0,1]\). The network architectures and quantum circuit configurations are chosen such that the resulting number of trainable parameters is comparable to ensure a fair comparison across the classical A-PINN, N-QPINN, and A-QPINN models. This setup allows for a direct evaluation of the effect of the proposed formulations on approximation accuracy.
\begin{table}[ht]
\centering
\caption{Performance of different methods for the 1D Volterra IDE system~\eqref{ex2-volterra-original}}
\begin{tabular}{m{2cm}cccc}
\toprule
Method 
& \multicolumn{1}{c}{Network} 
& \multicolumn{2}{c}{$\mathcal{L}_2$-error} 
& \multicolumn{1}{c}{Trainable} \\
\cmidrule(lr){3-4}
& Size & $u(x)$ & $v(x)$ & Parameters \\
\midrule
\shortstack{Classical\\A-PINN~\cite{Yuan2022APINN}}  & $6\times 4$ & $3.94 \pm 1.48 \times 10^{-4}$ & $5.29 \pm 2.57 \times 10^{-4}$ & $166$ \\[8pt]
\shortstack{N-QPINN\\(this work)}  & $5 \times 3$ & $1.02 \pm 1.25 \times 10^{-4}$ & $7.57 \pm 1.74 \times 10^{-5}$ & $52$ \\[8pt]
\shortstack{A-QPINN\\(this work)}  & $6 \times 3$ & $5.22 \pm 2.71 \times 10^{-4}$ & $1.45 \pm 1.35 \times 10^{-4}$ & $112$ \\
\bottomrule
\end{tabular}
\label{IDE-System-table}
\end{table}
The quantitative results are summarized in Table~\ref{IDE-System-table}. Both proposed QPINN approaches improve upon the classical A-PINN. In particular, N-QPINN achieves the highest accuracy while using significantly fewer trainable parameters, whereas A-QPINN also demonstrates competitive performance, especially for the second solution component. These results highlight the effectiveness of the proposed quantum-integrated formulations in enhancing predictive accuracy for nonlinear Volterra systems.
\begin{figure}[ht]
    \centering
    \includegraphics[width=1\linewidth]{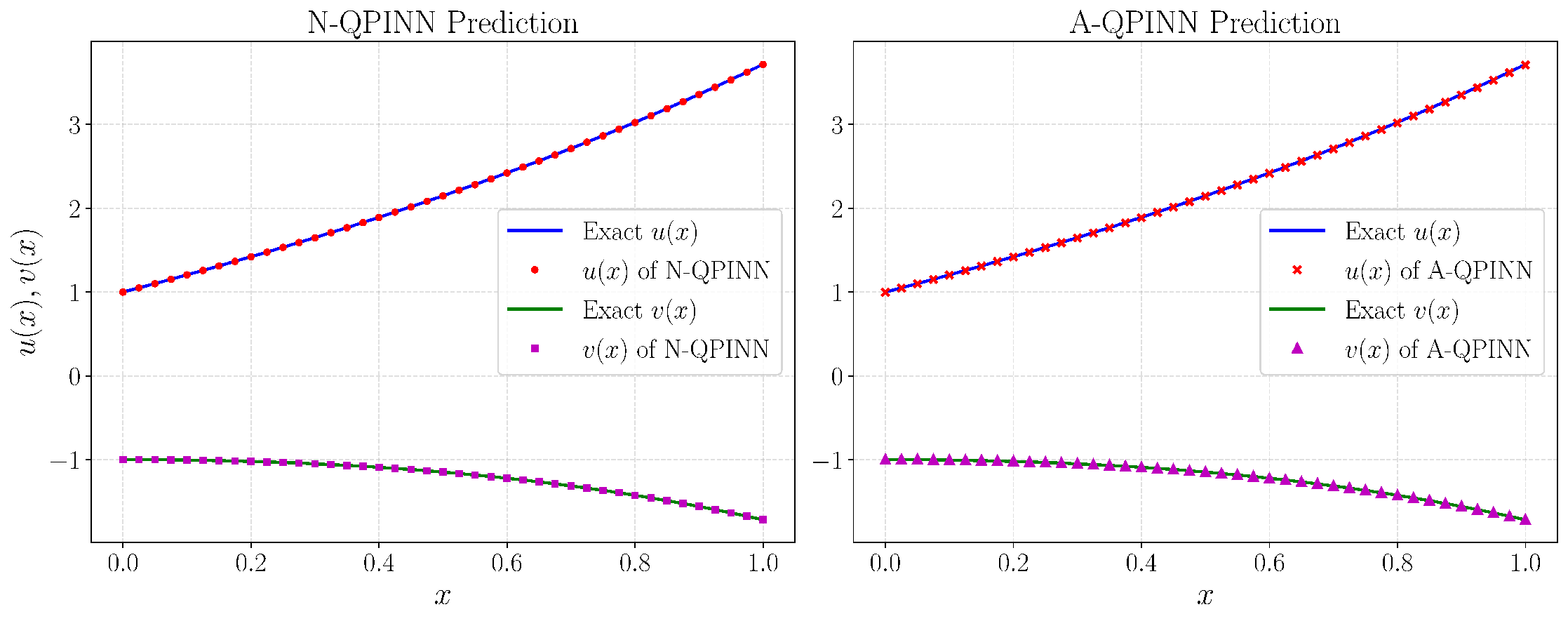}
    \caption{Predicted solutions of the 1D Volterra IDE system \eqref{ex2-volterra-original} using N-QPINN (left) and A-QPINN (right). The solid blue line denotes the exact solution, while red markers represent the network predictions.}
    \label{fig:Ex2-exactvspred}
\end{figure}
Figure~\ref{fig:Ex2-exactvspred} shows the predicted solutions of \(u(x)\) and \(v(x)\) obtained using both N-QPINN and A-QPINN compared with the exact analytical solutions. In each subplot, the solid curves represent the exact solutions, while the discrete markers correspond to the network predictions. Both quantum models accurately recover the qualitative behavior of the system, with N-QPINN exhibiting almost exact overlap with the analytical profiles across the entire interval. The A-QPINN predictions also match the exact solution closely, with slightly larger deviations near the boundary. These visual results are consistent with the quantitative errors reported in Table~\ref{IDE-System-table}.
\subsection{Two-dimensional nonlinear integro-partial differential equation}\label{Example-3}Consider the following nonlinear 2D Volterra integro-partial differential equation:
\begin{equation}\label{ex3-2dvolterra-original}
\left\{
\begin{aligned}
\frac{\partial^2 u(t,x)}{\partial t^2}
&=
\frac{\partial u(t,x)}{\partial x}-\frac{\partial u(t,x)}{\partial t}-u(t,x)+g(t,x)\\
&\qquad\qquad\qquad\qquad+\lambda \int_{0}^{x}\int_{0}^{t}
t\cos(y_1-y_2)\,u(y_1,y_2)\,dy_1dy_2, \\[6pt]
u(0,x) &= x, \qquad
u_t(0,x) = \sin(x), \\[4pt]
u(t,0) &= t\sin t, \qquad (t,x)\in[0,1]\times[0,1].
\end{aligned}
\right.
\end{equation}
Here, the parameter $\lambda$ is fixed as $\lambda=1$, and the forcing term $g(t,x)$ is chosen such that the exact solution is given by
\[
u(t,x)=x+t\sin(t+x).
\]
We solve the problem using the N-QPINN and A-QPINN approaches introduced earlier.  
For the N-QPINN formulation, the nonlocal Volterra operator is approximated numerically.  
At each collocation point $(t_j,x_j)$, the integral term
\[
\int_{0}^{x_j}\int_{0}^{t_j}
t_j\cos(y_1-y_2)\,u(y_1,y_2)\,dy_1dy_2
\]
is evaluated using a two-dimensional product trapezoidal quadrature rule.  
The integration domain $[0,t_j]\times[0,x_j]$ is partitioned into $M_t\times M_x$ uniform subintervals with step sizes
$h_t=t_j/M_t$ and $h_x=x_j/M_x$, and nodes
$y_1^{(m)}=m h_t$, $y_2^{(n)}=n h_x$.  
The quadrature approximation is given by
\[
\widehat{I}(t_j,x_j)
=
t_j
\sum_{m=0}^{M_t}\sum_{n=0}^{M_x}
\omega_m\,\omega_n\,
\cos\!\bigl(y_1^{(m)}-y_2^{(n)}\bigr)\,
u_\theta\!\bigl(y_1^{(m)},y_2^{(n)}\bigr),
\]
where $\omega_m$ and $\omega_n$ denote the trapezoidal weights. Substituting this approximation into \eqref{ex3-2dvolterra-original} yields the discrete residual
\[
\widehat{\mathcal{R}}(t_j,x_j)
=
u_{\theta,tt}(t_j,x_j)
-
\Bigl(
u_{\theta,x}(t_j,x_j)
-
u_{\theta,t}(t_j,x_j)
-
u_\theta(t_j,x_j)
+
g(t_j,x_j)
+
\lambda\,\widehat{I}(t_j,x_j)
\Bigr).
\]
The N-QPINN loss function consists of the mean-squared residual error over all collocation points together with the initial and boundary conditions.

To solve the same problem using the A-QPINN approach, we avoid numerical quadrature by introducing the auxiliary variable
\[
w(t,x)
=
\int_{0}^{x}\int_{0}^{t}
\cos(y_1-y_2)\,u(y_1,y_2)\,dy_1dy_2.
\]
Differentiating with respect to $t$ and $x$ yields the local relation
\[
\frac{\partial^2 w(t,x)}{\partial x\,\partial t}
=
\cos(t-x)\,u(t,x),
\qquad
w(0,x)=w(t,0)=0. 
\]
Using this auxiliary function, the integro-partial differential equation
\eqref{ex3-2dvolterra-original} can be equivalently rewritten as
\begin{equation}
\left\{
\begin{aligned}
\dfrac{\partial^2 u(t,x)}{\partial t^2}&=\dfrac{\partial u(t,x)}{\partial x}-\dfrac{\partial u(t,x)}{\partial t}-u(t,x)+g(t,x)+\lambda\, t\, w(t,x), \\[6pt]
\dfrac{\partial^2 w(t,x)}{\partial x\,\partial t}
&=\cos(t-x)\,u(t,x),\qquad w(0,x)=w(t,0)=0\\[6pt]
u(0,x)= x,&\;u_t(0,x)= \sin(x), \;
u(t,0) = t\sin t, \; (t,x)\in[0,1]\times[0,1].
\end{aligned}
\right.
\end{equation}
This coupled system is enforced using a multi-output quantum neural network. The A-QPINN loss function is defined as $\mathcal{L}_{\text{total}}=\mathcal{L}_{\text{IDE}}+\mathcal{L}_{\text{Aux}}+\mathcal{L}_{\text{BC}}+\mathcal{L}_{\text{IC}}.$

\begin{table}[ht]
\centering
\caption{The performance of different methods for nonlinear 2D Volterra integro-partial differential equation ~\eqref{ex3-2dvolterra-original}}
\begin{tabular}{lccc}
\toprule
Method & Size& $\mathcal{L}_2$-error &  Trainable parameters \\
\midrule
Classical A-PINN~\cite{Yuan2022APINN}  & $6\times 3$ &  $3.21 \pm 2.50\times 10^{-3}$ & $124$ \\
N-QPINN (This work)                    & $6\times 4$  & $1.45\pm2.15\times 10^{-3}$  & $97$ \\
A-QPINN (This work)                    & $6\times 3$  & $5.13\pm1.75\times 10^{-4}$  & $123$ \\
\bottomrule
\end{tabular}
\label{ex-3-table}
\end{table}
The performance for the nonlinear two-dimensional Volterra integro-PDE~\eqref{ex3-2dvolterra-original} is summarized in Table~\ref{ex-3-table}. The classical A-PINN provides a baseline level of accuracy, while both QPINN approaches demonstrate improved performance under comparable model sizes. In particular, A-QPINN achieves the highest accuracy, indicating the effectiveness of the auxiliary-function formulation for this class of nonlinear nonlocal problems. In contrast, N-QPINN yields relatively lower accuracy. 
\begin{figure}[ht]
    \centering
    \includegraphics[width=1\linewidth]{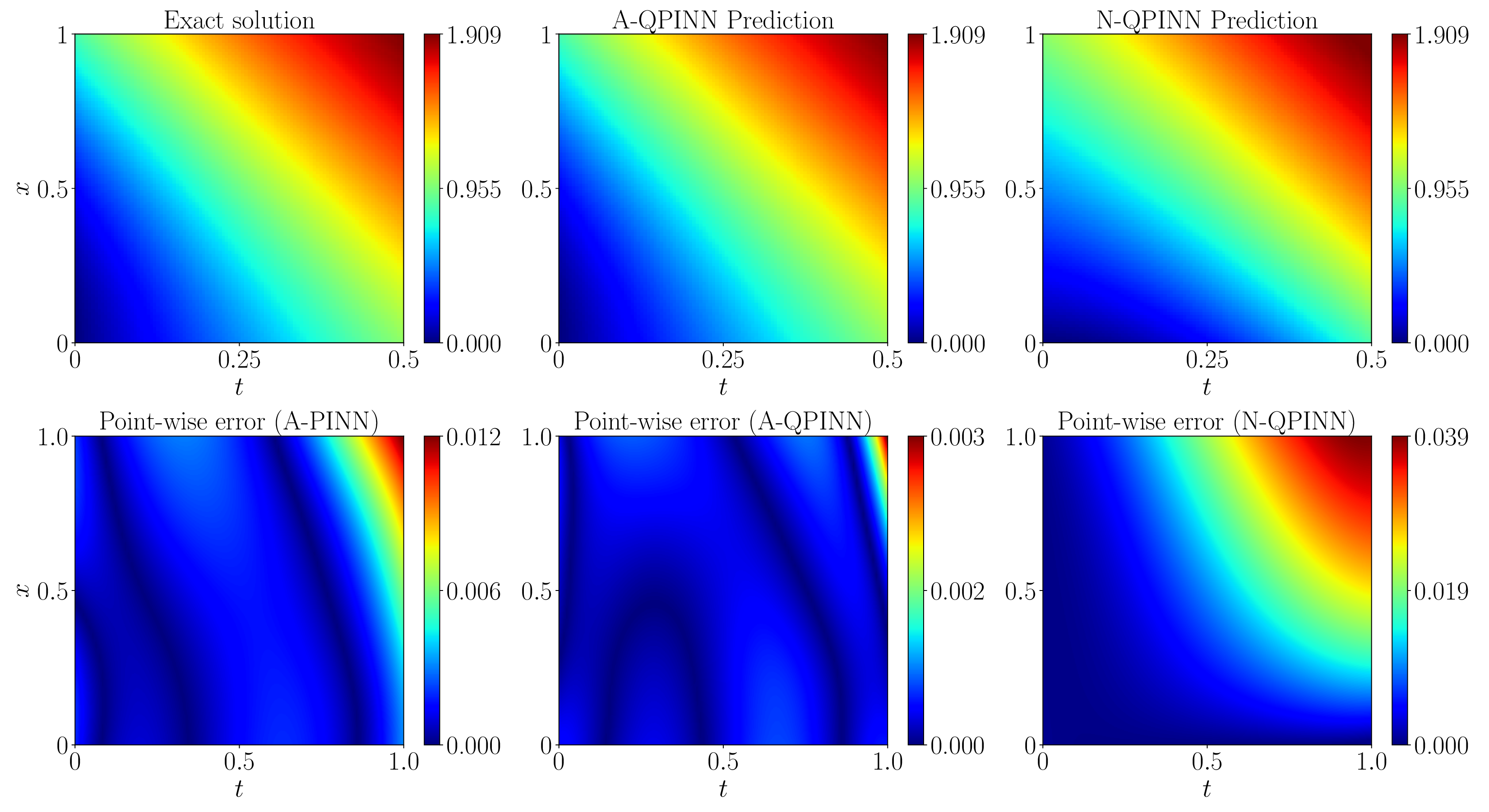}
    \caption{Top row: exact solution (left), A-QPINN prediction (middle), and N-QPINN prediction (right). Bottom row: corresponding point-wise absolute error distributions for the classical A-PINN, A-QPINN, and N-QPINN methods.}
    \label{fig:Ex3}
\end{figure}
Figure~\ref{fig:Ex3} presents the exact solution together with the predictions obtained using A-QPINN and N-QPINN, along with the corresponding pointwise error distributions. The A-QPINN solution exhibits close agreement with the exact profile across the domain, with uniformly small errors. In comparison, the N-QPINN prediction shows larger deviations, particularly in regions where nonlocal effects are more pronounced. These observations are consistent with the quantitative results in Table~\ref{ex-3-table}, further supporting the improved accuracy and robustness of the auxiliary-function-based formulation.
\subsection{nonlinear time-fractional integro–partial differential equation}\label{Example-4}
In this experiment, we investigate a nonlinear time-fractional integro–partial differential equation involving Caputo fractional derivatives, a fractional Volterra integral, diffusion, and a nonlinear reaction term. The governing equation is
\begin{equation}\label{ex4-governing}
\begin{aligned}
D_t^{\alpha_1} u(x,t)
+ D_t^{\alpha_2} u(x,t)
- \frac{\partial^2 u}{\partial x^2}(x,t)&+ \mathcal{I}_t^{\beta} u(x,t)
- u(x,t)^2
= f(x,t), \\
&\qquad\qquad\qquad (x,t)\in[0,1]\times[0,1].
\end{aligned}
\end{equation}
with initial and boundary conditions
\begin{equation}
    u(x,0) = 0,\quad x\in[0,1], 
    \qquad
    u(0,t)=0,\; u(1,t)=0,\quad t\in[0,1].
\end{equation}
The Caputo fractional derivative $D_t^{\alpha}$ and fractional integral $\mathcal{I}_t^{\beta}$ are defined in the equation~\eqref{def-capto} for $\alpha,\beta\in(0,1)$. The source term $f(x,t)$ is chosen so that the exact solution is $u(x,t)=t^{\,4+\alpha_1}\sin(\pi x$). In this example we set $\alpha_1=0.7$, $\alpha_2=0.4$, $\beta=0.2$, and use a temporal grading parameter $\rho=2$.
Here, we solve the integro–fractional model~\eqref{ex4-governing} exclusively using the N-QPINN approach, since the structure of the problem makes both A-PINN and A-QPINN inapplicable. The auxiliary-function strategy used in A-QPINN cannot be extended here because the Caputo operators depend on the full time history of $\partial_s u$, while the fractional integral introduces an additional memory effect. These operators cannot be rewritten as a finite-dimensional first-order system without introducing infinitely many auxiliary states. To implement N-QPINN, the fractional operators $D_t^{\alpha}u(x,t)$ and $\mathcal{I}_t^\beta u(x,t)$ are approximated on a graded temporal mesh,
\begin{equation*}
    t_n=\left(\frac{n}{N}\right)^{\rho}, \qquad n=0,\ldots,N,
\end{equation*}
which improves resolution near $t=0$ to capture the weak singularity inherent in fractional kernels. At each time level $t_n$, the Caputo derivatives are approximated using precomputed convolution-type coefficients
\begin{equation*}
    D_t^{\alpha}u(x,t_n)
    \approx \frac{1}{\Gamma(1-\alpha)}
    \sum_{k=0}^{n} c_{n,k}^{(\alpha)}\,u(x,t_k),
\end{equation*}
while the fractional integral is discretized as
\begin{equation*}
    \mathcal{I}_t^{\beta}u(x,t_n)
    \approx \sum_{k=0}^{n} q_{n,k}^{(\beta)}\,u(x,t_k).
\end{equation*}
The coefficients $c_{n,k}^{(\alpha)}$ and $q_{n,k}^{(\beta)}$ are generated once before training and stored as GPU tensors. For each collocation pair $(x_j,t_n)$, the model evaluations $u_\theta(x_j,t_k)$ for all $k\le n$ are used to assemble the discrete fractional operators. The resulting N-QPINN residual at $(x_j,t_n)$ is
\begin{equation*}
\begin{aligned}
    \widehat{\mathcal{R}}(x_j,t_n)
    &=
    D_t^{\alpha_1}u_\theta(x_j,t_n)
    +D_t^{\alpha_2}u_\theta(x_j,t_n)
    -u_{\theta,xx}(x_j,t_n)\\
    &\qquad\qquad\qquad+\mathcal{I}_t^{\beta}u_\theta(x_j,t_n)-u_\theta(x_j,t_n)^2
    -f(x_j,t_n).
    \end{aligned}
\end{equation*}
The loss function is given by
\begin{equation*}
\begin{aligned}
    \mathcal{L}_{\text{total}}
    =
    \frac{1}{N_R}\sum_{j,n}|\widehat{\mathcal{R}}(x_j,t_n)|^2
    &+\frac{1}{N_{\text{IC}}}\sum_{j}|u_\theta(x_j,0)|^2\\
    &+\frac{1}{N_{\text{BC}}}\sum_{n}\bigl(|u_\theta(0,t_n)|^2+|u_\theta(1,t_n)|^2\bigr).
    \end{aligned}
\end{equation*}
A hybrid quantum architecture with four variational layers and four qubits is employed, preceded and followed by classical linear transformations. 
The interior collocation set consists of 500 graded temporal points and 500 spatial points. The convolutional evaluation of the fractional operators relies on batched model predictions across all necessary historical time levels, enabling efficient training despite strong temporal nonlocality. To quantify the predictive accuracy, the relative $\mathcal{L}_2$-error obtained is $1.87\times10^{-2}$. 
Figure~\ref{fig:Ex4} presents a comparison between the predicted and exact solutions, showing that the learned solution closely follows the exact profile while the pointwise error remains small across most of the domain. 
This experiment demonstrates that N-QPINN can successfully handle nonlinear integro–fractional PDEs involving both Caputo derivatives and fractional Volterra integrals, whereas auxiliary-state formulations such as A-PINN and A-QPINN are not well suited for problems with such strong temporal nonlocality.
\begin{figure}[ht]
    \centering
    \includegraphics[width=1\linewidth]{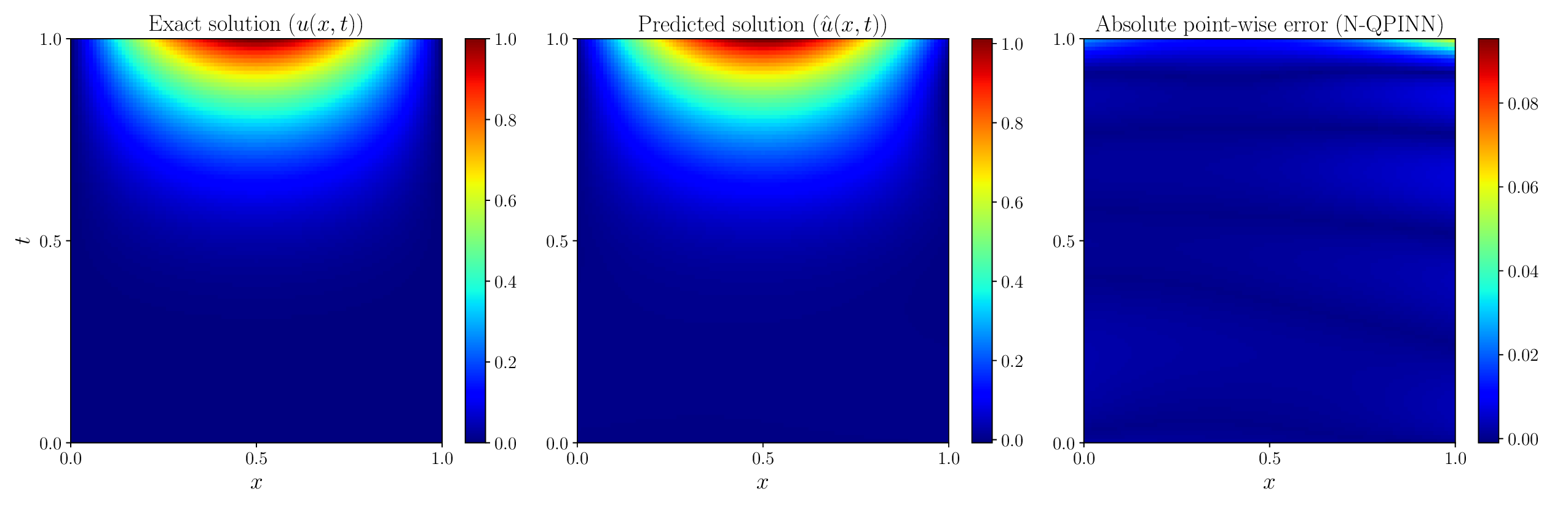}
    \caption{N-QPINN solution of the nonlinear time-fractional integro–partial differential equation~\eqref{ex4-governing}. 
    From left to right: exact solution, N-QPINN prediction, and pointwise absolute error.}
    \label{fig:Ex4}
\end{figure}
Figure~\ref{fig:Ex4} shows the exact solution, the N-QPINN prediction, and the pointwise error. The N-QPINN solution matches the analytical profile well across the domain, with larger discrepancies occurring only near the final time, where the combined effects of fractional memory and nonlinear dynamics are strongest.

\section{Conclusion}\label{sec_concl}
In this work, we developed a comprehensive quantum physics-informed neural network (QPINN) framework for solving nonlinear integro-differential equations and fractional integro–partial differential equations. Building on a constructed quantum neural network architecture, we established approximation results showing that the proposed QNN generates finite trigonometric representations with explicit $\mathcal{L}_2$ error bounds, thereby providing a solid theoretical foundation for its use in physics-informed learning. Based on this framework, we introduced two complementary formulations for handling nonlocal operators. N-QPINN incorporates integral and fractional operators through high-order numerical quadrature while retaining automatic differentiation for local differential terms, enabling a unified treatment of nonlinear IDEs, coupled IDE systems, and FIPDEs. In contrast, A-QPINN avoids numerical quadrature by introducing auxiliary fields that reformulate integro-differential equations as equivalent coupled PDE systems, allowing the constraints to be enforced through multi-output QNNs. A series of numerical experiments demonstrated that both QPINN formulations consistently outperform classical physics-informed neural networks under comparable configurations. The results further indicate that the relative performance of N-QPINN and A-QPINN depends on the structure and dimensionality of the underlying nonlocal operator: N-QPINN is particularly effective for problems involving complex or fractional nonlocal terms, while A-QPINN can achieve superior accuracy for certain multidimensional Volterra-type equations by eliminating quadrature-induced errors. Overall, this work establishes quantum physics-informed neural networks as a viable and theoretically grounded approach for solving complex nonlocal differential equations. The proposed QPINN framework opens several promising directions for future research, including applications in viscoelasticity, anomalous diffusion, heat transfer with memory effects, and nonlocal fluid dynamics, the incorporation of noise-aware quantum training strategies, and deployment on emerging quantum hardware.

\section*{Declarations}
\subsection*{Ethical Approval}
The authors confirm that the present work adheres to the ethical guidelines of the journal. This work has not been published elsewhere and is not under consideration by any other publication.
\subsection*{Funding}
The authors did not receive any specific funding for this work.

\subsection*{Author Information}
Deepak Gupta${}^a$, Himanshu Pandey${}^a$, Ratikanta Behera${}^a$ \\[0.5em]
${}^a$Department of Computational and Data Sciences, Indian Institute of Science, Bangalore, India \\[0.5em]
\textbf{Corresponding author:} Ratikanta Behera

		
\section*{Acknowledgements}
The authors would like to acknowledge the Department of Computational and Data Sciences, Indian Institute of Science, Bangalore, for providing the research environment and facilities necessary to carry out this work. 

\section*{Author Contributions}
\begin{itemize}
    \item \textbf{Deepak Gupta:} Writing - original draft, investigation, methodology, coding. 
    \item \textbf{Ratikanta Behera:} Supervision, visualization, validation, writing review and editing.
\end{itemize}

\bibliographystyle{unsrt}
\bibliography{Ref_DG}	
\end{document}